\documentclass[12pt,a4paper,reqno]{amsart}
\usepackage[a4paper,left=16mm,right=16mm,top=36mm,bottom=36mm]{geometry}

\usepackage{color}
\usepackage{amssymb}
\usepackage{amsmath}
\usepackage{graphicx}
\usepackage{float}
\usepackage[utf8]{inputenc}
\usepackage[english]{babel}
\usepackage{stmaryrd}
\usepackage{mathtools}
\usepackage{subcaption}
\usepackage{pgfplots}
\usepackage{grffile}
\pgfplotsset{compat=newest}
\usetikzlibrary{plotmarks}
\usetikzlibrary{arrows.meta}
\usepgfplotslibrary{patchplots}

\definecolor{mycolor1}{rgb}{0.46600,0.67400,0.18800}%
\definecolor{mycolor2}{rgb}{0.30100,0.74500,0.93300}%

\DeclarePairedDelimiter\floor{\lfloor}{\rfloor}
\newcommand{\R}{\mathbb{R}}
\newcommand{\Z}{\mathbb{Z}}
\newcommand{\N}{\mathbb{N}}
\newcommand{\eps}{\varepsilon}
\newcommand{\fhi}{\varphi}
\newcommand{\weak}{\rightharpoonup}

\newcommand{\vertiii}[1]{{\left\vert\kern-0.25ex\left\vert\kern-0.25ex\left\vert #1
    \right\vert\kern-0.25ex\right\vert\kern-0.25ex\right\vert}}

\def\calO{\mathcal{O}}

\def\calT{\mathcal{T}}

\def\XXint#1#2#3{{\setbox0=\hbox{$#1{#2#3}{\int}$}
     \vcenter{\hbox{$#2#3$}}\kern-.5\wd0}}

\numberwithin{equation}{section}

\usepackage[]{hyperref}
\usepackage[noabbrev, capitalise, nameinlink]{cleveref}
\crefname{equation}{}{}
\crefname{assumption}{Assumption}{Assumptions}
\crefname{rmrk}{Remark}{Remarks}
\crefformat{equation}{\textup{#2(#1)#3}}

\usepackage[colorinlistoftodos,prependcaption]{todonotes}

\newtheorem{dfntn}{Definition}[section] 
\newtheorem{rmrk}{Remark}[section]    
\newtheorem{thrm}{Theorem}[section]     
\newtheorem{lmm}{Lemma}[section]
\newtheorem{cor}{Corollary}[section]

\begin{document}

\title[Homogenization of nondivergence-form equations with Cordes coefficients]{Homogenization of nondivergence-form elliptic equations with discontinuous coefficients and finite element approximation of the homogenized problem}

\author[T. Sprekeler]{Timo Sprekeler}
\address{National University of Singapore, Department of Mathematics, 10 Lower Kent Ridge Road, Singapore 119076, Singapore.}
\email{timo.sprekeler@nus.edu.sg}

\subjclass[2010]{35B27, 35J15, 65N12, 65N30}
\keywords{Homogenization, nondivergence-form elliptic PDE, Cordes condition, finite element methods}
\date{\today}

\begin{abstract}
We study the homogenization of the equation $-A(\frac{\cdot}{\eps}):D^2 u_{\eps} = f$ posed in a bounded convex domain $\Omega\subset \R^n$ subject to a Dirichlet boundary condition and the numerical approximation of the corresponding homogenized problem, where the measurable, uniformly elliptic, periodic and symmetric diffusion matrix $A$ is merely assumed to be essentially bounded and (if $n>2$) to satisfy the Cordes condition. In the first part, we show existence and uniqueness of an invariant measure by reducing to a Lax--Milgram-type problem, we obtain $L^2$-bounds for periodic problems in double-divergence-form, we prove homogenization under minimal regularity assumptions, and we generalize known corrector bounds and results on optimal convergence rates from the classical case of H\"{o}lder continuous coefficients to the present case. In the second part, we suggest and rigorously analyze an approximation scheme for the effective coefficient matrix and the solution to the homogenized problem based on a finite element method for the approximation of the invariant measure, and we demonstrate the performance of the scheme through numerical experiments.  
\end{abstract}

\maketitle

\section{Introduction}

In this paper, we discuss the homogenization of the linear elliptic nondivergence-form problem
\begin{align}\label{intro 1}
\begin{split}
-A\left(\frac{\cdot}{\eps}\right):D^2 u_{\eps} &= f\quad\text{in }\Omega,\\ u_{\eps} &= g\quad\text{on }\partial\Omega,
\end{split}
\end{align}
where $\eps > 0$ is a small parameter, $\Omega\subset \R^n$ is a bounded convex domain, $f\in L^2(\Omega)$, $g\in H^2(\Omega)$, and $A\in L^{\infty}(\R^n;\R^{n\times n}_{\mathrm{sym}})$ is $\Z^n$-periodic, uniformly elliptic, and (if $n>2$) satisfies the Cordes condition \eqref{Cordes} which dates back to \cite{Cor56}. In this setting, it is known that there exists a unique solution $u_{\eps}\in H^2(\Omega)$ to \eqref{intro 1}; see \cite{SS13,Tal65}. We study homogenization of the problem \eqref{intro 1} as well as the numerical approximation of the corresponding homogenized problem.

The theory of periodic homogenization for \eqref{intro 1} is classical and well-understood if $A$ is H\"{o}lder continuous and $f,g$ and $\partial\Omega$ are sufficiently regular. In this case, it is known that as $\eps\searrow 0$ we have that $(u_{\eps})_{\eps > 0}$ converges in a suitable sense to the solution of the homogenized problem, i.e., the linear elliptic constant-coefficient problem
\begin{align}\label{intro u}
\begin{split}
-\bar{A}:D^2 u &= f\quad\text{in }\Omega,\\ u &= g\quad\text{on }\partial\Omega,
\end{split}
\end{align}
where the effective coefficient matrix $\bar{A} := \int_Y rA$ is obtained by integrating $A$ against an invariant measure $r$, defined as the solution to the periodic problem
\begin{align}\label{intro r}
-D^2:(rA) = 0\quad\text{in }Y,\qquad r \text{ is $Y$-periodic},\qquad \int_Y r = 1,
\end{align}
where $Y:=(0,1)^n$; see e.g., \cite{AL89,BLP11,ES08,JKO94}. Further, optimal convergence rates and corrector bounds in suitable norms are available in the literature; see e.g., \cite{GST22,GTY20,KL16,ST21}. For recent developments in stochastic homogenization of nondivergence-form problems, we refer to \cite{AFL22,AL17,AS14,GT22,GT23} and the references therein. It seems that the case of measurable diffusion matrices that are merely essentially bounded has not been studied yet and our first goal of this paper is to generalize various qualitative and quantitative homogenization results to this setting.

Concerning the development of numerical methods to approximate the solution to \eqref{intro 1}, it is well-known that, for small positive $\eps$, standard finite element methods require a very fine mesh that resolves the oscillations of the coefficients to deliver accurate approximations. Therefore, multiscale finite element methods for nondivergence-form problems have been developed in recent years; see \cite{CSS20,FGP23} for linear problems and \cite{GSS21,KS22} for Hamilton--Jacobi--Bellman (HJB) and Isaacs problems. For some results regarding finite difference schemes we refer to \cite{CM09,FO18,FO09}.

In particular, \cite{FGP23} suggests a multiscale finite element scheme based on the methodology of localized orthogonal decomposition (LOD) for linear nondivergence-form equations with essentially bounded measurable coefficients satisfying the Cordes condition. In \cite{CSS20}, a multiscale finite element scheme for \eqref{intro 1} with $A\in W^{1,q}_{\mathrm{per}}(Y;\R^{n\times n}_{\mathrm{sym}})$, $q>n$, and satisfying the Cordes condition has been suggested which is based on an approximation of the effective coefficient matrix via a finite element approximation of the invariant measure, relying on rewriting \eqref{intro r} in divergence-form. Since the latter is not possible in our setting, the second goal of this paper is to provide and analyze a finite element method for the approximation of the invariant measure in our case of merely essentially bounded measurable coefficients that satisfy the Cordes condition. 

The structure of this paper is as follows.

In Section \ref{Sec: Framework}, we provide the framework, covering the main assumptions, well-posedness of \eqref{intro 1}, and a uniform $H^2$-bound for the solution to \eqref{intro 1}. In Section \ref{Sec: Hom}, we study homogenization of \eqref{intro 1}. We begin by introducing a bilinear form, a coercivity property, and useful estimates that are used throughout the paper in Section \ref{Sec: Prelim}. In Section \ref{Sec: Per problems}, we analyze linear periodic problems in double-divergence-form and nondivergence-form. In particular, we find that any diffusion matrix $A$ that satisfies the assumptions stated below \eqref{intro 1} has a unique invariant measure $r\in L^2_{\mathrm{per}}(Y)$ and that 
\begin{align}\label{r form}
r = C \frac{\mathrm{tr}(A)}{\lvert A\rvert^2} (1-\Delta \psi),
\end{align}
where $C$ is a positive constant and $\psi$ is the unique solution to a Lax--Milgram-type problem in the subspace of $H^2_{\mathrm{per}}(Y)$ consisting of functions with mean zero (see Theorem \ref{Thm: invmeas}). In Section \ref{Sec: conv result}, we show that $(u_{\eps})_{\eps>0}$ converges weakly in $H^2(\Omega)$ to the unique solution $u\in H^2(\Omega)$ to the homogenized problem \eqref{intro u}, using the transformation argument from \cite{AL89}. Thereafter, in Sections \ref{Sec: corrector est} and \ref{Sec: Type}, we obtain $H^2$ corrector estimates and generalize some results on type-$\eps^2$ diffusion matrices (see Definition \ref{Def: type}) obtained in \cite{GST22} for H\"{o}lder continuous diffusion matrices to our setting.

In Section \ref{Sec: Num methods}, we study the numerical approximation of the homogenized problem via a novel finite element method for the approximation of the invariant measure based on \eqref{r form}. We perform a rigorous error analysis and demonstrate the theoretical results in numerical experiments provided in Section \ref{Sec: Exp}. In Section \ref{Sec: Pfs}, we collect the proofs of the results contained in this work.

\section{Framework}\label{Sec: Framework}

Let $\Omega\subset \R^n$ be a bounded convex domain. Let $f\in L^2(\Omega)$ and $g\in H^2(\Omega)$. For a small parameter $\eps>0$, we consider the problem 
\begin{align}\label{ueps problem}
\begin{split}
L_{\eps}u_{\eps}:=-A^{\eps}:D^2 u_{\eps} &= f\quad\text{in }\Omega,\\ u_{\eps} &= g\quad\text{on }\partial\Omega,
\end{split}
\end{align}
where $A^{\eps}:=A\left(\frac{\cdot}{\eps}\right)$ and $A\in \mathcal{M}(\lambda,\Lambda)$ for some constants $\lambda,\Lambda>0$. Here, we define 
\begin{align*}
\mathcal{M}(\lambda,\Lambda):=\left\{A\in L^{\infty}(\R^n;\R^{n\times n}_{\mathrm{sym}})\,\left\lvert A\text{ is $Y$-periodic and }\forall\xi\in \R^n\backslash\{0\}:\lambda \leq \frac{A\xi\cdot\xi}{\lvert \xi\rvert^2}\leq\Lambda\text{ a.e. in }\R^n\right.\right\},
\end{align*}
where $Y:= (0,1)^n$. We further assume that $A$ satisfies the Cordes condition (dating back to \cite{Cor56}), that is,
\begin{align}\label{Cordes}
\exists\,\delta\in (0,1]:\qquad\frac{\lvert A\rvert^2}{(\mathrm{tr}(A))^2} \leq \frac{1}{n-1+\delta}\quad\text{a.e. in }\R^n,
\end{align}
where $\lvert A\rvert:=\sqrt{A:A}$. Note that this is no restriction for dimensions $n\leq 2$:
\begin{rmrk}\label{Rk: Cordes in 1D,2D}
Let $A\in \mathcal{M}(\lambda,\Lambda)$ for some $\lambda,\Lambda > 0$. If $n\in\{1,2\}$, the Cordes condition \eqref{Cordes} holds:
\begin{itemize}
\item[(i)] If $n = 1$, we have that \eqref{Cordes} holds with $\delta = 1$.
\item[(ii)] If $n = 2$, we have that \eqref{Cordes} holds with $\delta = \frac{\lambda}{\Lambda}$. Indeed, this can be seen by writing $\lvert A\rvert^2$ as sum of the squared eigenvalues of $A$ and $\mathrm{tr}(A)$ as sum of the eigenvalues of $A$, and using that $\frac{s^2 + t^2}{(s+t)^2}\leq(1+\frac{s}{t})^{-1}$ for any $0<s\leq t$.
\end{itemize}
\end{rmrk}

The Cordes condition \eqref{Cordes} guarantees the existence of a function $\gamma\in L^{\infty}_{\mathrm{per}}(Y)$ such that $\gamma > 0$ and $\lvert \gamma A -I_n\rvert \leq \sqrt{1-\delta}$ almost everywhere, where $I_n$ denotes the identity matrix in $\R^{n\times n}$.

\begin{rmrk}\label{Rk: Cordes holds}
Let $A\in \mathcal{M}(\lambda,\Lambda)$ for some $\lambda,\Lambda > 0$ and suppose that \eqref{Cordes} holds. Then, the function
\begin{align*}
\gamma:=\frac{\mathrm{tr}(A)}{\lvert A \rvert^2}\in L^{\infty}(\R^n)
\end{align*}
is $Y$-periodic, satisfies $\frac{\lambda}{\Lambda^2}\leq \gamma\leq \frac{\Lambda}{\lambda^2}$ almost everywhere, and there holds $\lvert \tilde{A} - I_n\rvert^2 = n-\frac{(\mathrm{tr}(A))^2}{\lvert A \rvert^2} \leq 1-\delta$ almost everywhere, where $\tilde{A}:=\gamma A$. Note that $\tilde{A}\in \mathcal{M}(\frac{\lambda^2}{\Lambda^2},\frac{\Lambda^2}{\lambda^2})$.
\end{rmrk} 
It is then known (see \cite{SS13,Tal65}) that \eqref{ueps problem} has a unique solution in $H^2(\Omega)$. Further, we can obtain a uniform bound on the $H^2$-norm of the solution.

\begin{thrm}[well-posedness and uniform $H^2$-bound]\label{Thm: Well-posedness of uepsproblem}
Let $\Omega\subset\R^n$ be a bounded convex domain. Let $A\in \mathcal{M}(\lambda,\Lambda)$,  $f\in L^2(\Omega)$, $g\in H^2(\Omega)$, and suppose that \eqref{Cordes} holds. Then, for any $\eps > 0$ there exists a unique solution $u_{\eps}\in H^2(\Omega)$ to \eqref{ueps problem}, and we have the bound
\begin{align}\label{uniform H2 bd}
\|u_{\eps}\|_{H^2(\Omega)}\leq C (\|f\|_{L^2(\Omega)}+\|g\|_{H^2(\Omega)})
\end{align}
for some constant $C = C(\mathrm{diam}(\Omega),\lambda,\Lambda,n,\delta) > 0$.
\end{thrm}

\section{Homogenization}\label{Sec: Hom}

\subsection{Preliminaries}\label{Sec: Prelim}

Before we start with the main discussion we briefly highlight an important observation which will be crucial for this work. We write $L^2_{\mathrm{per},0}(Y):=\{\fhi\in L^2_{\mathrm{per}}(Y):\int_Y \fhi = 0\}$ and $H^k_{\mathrm{per},0}(Y):=\{\fhi\in H^k_{\mathrm{per}}(Y):\int_Y \fhi = 0\}$ for $k\in \N$.

\begin{lmm}[the bilinear form $b_{\mu}$]\label{Lmm: b_mu}
Let $A\in \mathcal{M}(\lambda,\Lambda)$ for some $\lambda,\Lambda > 0$ and suppose that \eqref{Cordes} holds. Let $\gamma\in L^{\infty}_{\mathrm{per}}(Y)$ be defined as in Remark \ref{Rk: Cordes holds}. Then, for any $\mu\geq 0$, the bilinear form
\begin{align*}
b_{\mu}: H^2_{\mathrm{per}}(Y)\times H^2_{\mathrm{per}}(Y)\rightarrow \R,\qquad b_{\mu}(\fhi_1,\fhi_2):=(\mu \fhi_1-\gamma A:D^2 \fhi_1,\mu \fhi_2 -\Delta \fhi_2)_{L^2(Y)}
\end{align*}
satisfies the bound
\begin{align*}
\mu^2\|\fhi\|_{L^2(Y)}^2 + 2\mu \|\nabla \fhi\|_{L^2(Y)}^2 + \|D^2 \fhi\|_{L^2(Y)}^2 =   \|\mu \fhi - \Delta \fhi\|_{L^2(Y)}^2\leq C_{\delta}\, b_{\mu}(\fhi,\fhi)\qquad\forall \fhi\in H^2_{\mathrm{per}}(Y),
\end{align*} 
where $C_{\delta}:=(1-\sqrt{1-\delta})^{-1}>0$. In particular, if $\mu > 0$ we have that $b_{\mu}$ is coercive on $H^2_{\mathrm{per}}(Y)$, and if $\mu = 0$ we have that $b_{\mu}$ is coercive on $H^2_{\mathrm{per},0}(Y)$.
\end{lmm}

The following inequalities will be used frequently:

\begin{rmrk}\label{Rk: norm on H^2_per,0}
For any $v\in  H^1_{\mathrm{per},0}(Y)$ we have the Poincar\'{e} inequality $\|v\|_{L^2(Y)}\leq \frac{\sqrt{n}}{\pi}\|\nabla v\|_{L^2(Y)}$; see e.g., \cite{Beb03}. Further, for any $\fhi\in H^2_{\mathrm{per}}(Y)$ we have the identity $\|D^2 \fhi\|_{L^2(Y)} = \|\Delta\fhi\|_{L^2(Y)}$. In particular, there holds $\frac{\pi^2}{n}\|\fhi\|_{L^2(Y)}\leq \frac{\pi}{\sqrt{n}}\|\nabla \fhi\|_{L^2(Y)}\leq \|D^2 \fhi\|_{L^2(Y)} = \|\Delta\fhi\|_{L^2(Y)}$ for any $\fhi\in H^2_{\mathrm{per},0}(Y)$.
\end{rmrk}

Using these results, we will now study periodic problems in double-divergence-form and nondivergence-form.

\subsection{Periodic problems and invariant measures}\label{Sec: Per problems}

In this section, we discuss the periodic double-divergence-form problem
\begin{align}\label{q equation}
-D^2:(qA) = f\quad\text{in }Y,\qquad q\text{ is Y-periodic},
\end{align} 
and the periodic nondivergence-form problem
\begin{align}\label{v equation}
-A:D^2 v = f\quad\text{in }Y,\qquad v\text{ is Y-periodic},
\end{align}
where $f\in L^2_{\mathrm{per}}(Y)$, $A\in \mathcal{M}(\lambda,\Lambda)$ for some $\lambda,\Lambda > 0$, and we assume that the Cordes condition \eqref{Cordes} holds. We seek solutions  $q\in L^2_{\mathrm{per}}(Y)$ to \eqref{q equation}, i.e., $(q,-A:D^2 \fhi)_{L^2(Y)} = (f,\fhi)_{L^2(Y)}$ for any $\fhi\in H^2_{\mathrm{per}}(Y)$. First, we introduce the notion of an invariant measure; see also \cite{BLP11,BS17}.

\begin{dfntn}[invariant measure]
Let $A\in \mathcal{M}(\lambda,\Lambda)$ for some $\lambda,\Lambda > 0$. A function $r\in L^2_{\mathrm{per}}(Y)$ is called an invariant measure to $A$ if $\int_Y r = 1$ and $(r,-A:D^2 \fhi)_{L^2(Y)} = 0$ for all $\fhi\in H^2_{\mathrm{per}}(Y)$. 
\end{dfntn}

Our first result is the existence and uniqueness of an invariant measure $r\in L^2_{\mathrm{per}}(Y)$ to any $A\in \mathcal{M}(\lambda,\Lambda)$ satisfying the Cordes condition, and an $L^2$-bound.

\begin{thrm}[existence, uniqueness and properties of invariant measures]\label{Thm: invmeas}
Let $A\in \mathcal{M}(\lambda,\Lambda)$ for some $\lambda,\Lambda > 0$ and suppose that \eqref{Cordes} holds. Let $\gamma\in L^{\infty}_{\mathrm{per}}(Y)$ be defined as in Remark \ref{Rk: Cordes holds} and set $\tilde{A}:=\gamma A$. Let $b_0:H^2_{\mathrm{per}}(Y)\times H^2_{\mathrm{per}}(Y)\rightarrow \R$ and $C_{\delta} > 0$ be defined as in Lemma \ref{Lmm: b_mu}. Then, we have the following:
\begin{itemize}
\item[(i)] There exists a unique $\psi\in H^2_{\mathrm{per},0}(Y)$ such that $b_0(\fhi,\psi) = \int_Y \tilde{A}:D^2 \fhi$ for any $\fhi\in H^2_{\mathrm{per},0}(Y)$ and we have the bound $\|\Delta \psi\|_{L^2(Y)}\leq \sqrt{n}\frac{\Lambda}{\lambda}C_{\delta}$. 

\item[(ii)] The function $\tilde{r}:=1- \Delta \psi\in L^2_{\mathrm{per}}(Y)$ is the unique invariant measure to $\tilde{A}$, and there holds $\tilde{r}\geq 0$ almost everywhere.
\item[(iii)] The function $r := (\gamma,\tilde{r})_{L^2(Y)}^{-1} \gamma\tilde{r}\in L^2_{\mathrm{per}}(Y)$ is the unique invariant measure to $A$, and there holds $r\geq 0$ almost everywhere. Further, we have the bound
\begin{align}\label{r bound}
\|r\|_{L^2(Y)} \leq \frac{\Lambda^3}{\lambda^3}  \|\tilde{r}\|_{L^2(Y)}\leq \frac{\Lambda^3}{\lambda^3}\left(  \sqrt{n}\frac{\Lambda}{\lambda}C_{\delta} + 1 \right).
\end{align} 
\end{itemize}
\end{thrm}

With Theorem \ref{Thm: invmeas} at hand, we can now state the main result on problems of the form \eqref{q equation}.

\begin{thrm}[analysis of the problem \eqref{q equation}]\label{Thm: Main thm double div}
Let $f\in L^2_{\mathrm{per}}(Y)$, $A\in \mathcal{M}(\lambda,\Lambda)$ for some $\lambda,\Lambda > 0$, and suppose that \eqref{Cordes} holds. Let $r\in L^2_{\mathrm{per}}(Y)$ denote the unique invariant measure to $A$, and let $C_{\delta} > 0$ be defined as in Lemma \ref{Lmm: b_mu}. Then, we have the following:
\begin{itemize}
\item[(i)] There exists a solution $q\in L^2_{\mathrm{per}}(Y)$ to the problem \eqref{q equation} if, and only if, $f\in L^2_{\mathrm{per},0}(Y)$.
\item[(ii)] If $f\in L^2_{\mathrm{per},0}(Y)$ and $q_1,q_2\in L^2_{\mathrm{per}}(Y)$ are solutions to \eqref{q equation}, then $q_1 - q_2 = cr$ with $c:=\int_Y (q_1-q_2)$. In particular, the problem \eqref{q equation} has a unique solution if $\int_Y q$ is prescribed. 
\item[(iii)] If $f\in L^2_{\mathrm{per},0}(Y)$, then the unique solution $q_0\in L^2_{\mathrm{per},0}(Y)$ to \eqref{q equation} subject to $\int_Y q = 0$, whose existence and uniqueness follows from (i)-(ii), satisfies the bound
\begin{align*}
\|q_0\|_{L^2(Y)}\leq \frac{n\Lambda}{\pi^2\lambda^2}C_{\delta} \left(1+\|r\|_{L^2(Y)}\right)\|f\|_{L^{2}(Y)}. 
\end{align*}
\end{itemize}
\end{thrm}

Next, we turn to the analysis of the problem \eqref{v equation}.

\begin{thrm}[analysis of the problem \eqref{v equation}]\label{Thm: per nondiv}
Let $f\in L^2_{\mathrm{per}}(Y)$, $A\in \mathcal{M}(\lambda,\Lambda)$ for some $\lambda,\Lambda > 0$, and suppose that \eqref{Cordes} holds. Let $\gamma\in L^{\infty}_{\mathrm{per}}(Y)$ be defined as in Remark \ref{Rk: Cordes holds} and let $b_0:H^2_{\mathrm{per}}(Y)\times H^2_{\mathrm{per}}(Y)\rightarrow \R$ and $C_{\delta} > 0$ be defined as in Lemma \ref{Lmm: b_mu}. Let $r\in L^2_{\mathrm{per}}(Y)$ denote the unique invariant measure to $A$. Then, we have the following:
\begin{itemize}
\item[(i)] There exists a solution $v\in H^2_{\mathrm{per}}(Y)$ to the problem \eqref{v equation} if, and only if, $(f,r)_{L^2(Y)} = 0$. Moreover, if $(f,r)_{L^2(Y)} = 0$, we have that $v\in H^2_{\mathrm{per}}(Y)$ is a solution to \eqref{v equation} if, and only if, $b_0(v,\fhi) = -(\gamma f,\Delta \fhi)_{L^2(Y)}$ for any $\fhi \in H^2_{\mathrm{per}}(Y)$.
\item[(ii)] If $(f,r)_{L^2(Y)} = 0$ and $v_1,v_2\in H^2_{\mathrm{per}}(Y)$ are solutions to \eqref{v equation}, then $v_1 - v_2 = \mathrm{const.}$ almost everywhere. In particular, the problem \eqref{v equation} has a unique solution if $\int_Y v$ is prescribed. 
\item[(iii)] If $(f,r)_{L^2(Y)} = 0$, then the unique solution $v_0\in H^2_{\mathrm{per},0}(Y)$ to \eqref{v equation} subject to $\int_Y v_0 = 0$,
whose existence and uniqueness follows from (i)-(ii), satisfies the bound
\begin{align*}
\|\Delta v_0\|_{L^2(Y)} \leq \frac{\Lambda}{\lambda^2} C_{\delta}\|f\|_{L^2(Y)}.
\end{align*}
\end{itemize}
\end{thrm}

We conclude this section by noting that the results regarding existence and uniqueness of solutions to \eqref{q equation} and \eqref{v equation} can also be obtained via the Fredholm alternative, observing that $K:L^2_{\mathrm{per}}(Y)\rightarrow L^2_{\mathrm{per}}(Y)$, $f\mapsto (-\gamma A:D^2+\mu\, \mathrm{id})^{-1}f$ is a compact linear operator for $\mu>0$ due to Lemma \ref{Lmm: b_mu}. Our approach however is a constructive one and forms the basis for a simple construction of finite element methods for the numerical approximation of \eqref{q equation} and \eqref{v equation}.

\subsection{The convergence result}\label{Sec: conv result}

Let $\Omega\subset \R^n$ be a bounded convex domain, $A\in \mathcal{M}(\lambda,\Lambda)$ for some $\lambda,\Lambda > 0$, $f\in L^2(\Omega)$, $g\in H^2(\Omega)$, and suppose that \eqref{Cordes} holds. We consider the problem 
\begin{align}\label{ueps problem again}
\begin{split}
L_{\eps}u_{\eps}:=-A^{\eps}:D^2 u_{\eps} &= f\quad\text{in }\Omega,\\ u_{\eps} &= g\quad\text{on }\partial\Omega,
\end{split}
\end{align}  
where $\eps > 0$ is a small positive parameter. We recall from Theorem \ref{Thm: Well-posedness of uepsproblem} that there exists a unique solution $u_{\eps}\in H^2(\Omega)$ to \eqref{ueps problem again} and that we have the uniform $H^2$-bound \eqref{uniform H2 bd}. Thus, there exists a function $u\in H^2(\Omega)$ with $u - g\in  H^1_0(\Omega)$ such that, upon passing to a subsequence,
\begin{align}\label{limit for subs}
u_{\eps}\weak u\quad\text{weakly in }H^2(\Omega)\text{ as }\eps\searrow 0,\qquad
u_{\eps}\rightarrow u\quad\text{strongly in }H^1(\Omega)\text{ as }\eps\searrow 0.
\end{align}
We denote the invariant measure to $A$ given by Theorem \ref{Thm: invmeas} by $r\in L^2_{\mathrm{per}}(Y)$ and multiply the equation $L_{\eps}u_{\eps} = f$, which holds a.e. in $\Omega$, by $r^{\eps}:=r(\frac{\cdot}{\eps})$ to obtain
\begin{align}\label{weighted with r}
-[rA]^{\eps}:D^2 u_{\eps} = r^{\eps} f \quad \text{a.e. in }\Omega,
\end{align}
where $[rA]^{\eps}:= r^{\eps}A^{\eps}$. Next, we apply the transformation argument from \cite{AL89} (see also \cite{JZ23}) to obtain a divergence-form equation for $u_{\eps}$. 

Noting that $rA\in L^2_{\mathrm{per}}(Y;\R^{n\times n}_{\mathrm{sym}})$, we introduce $\phi_j\in H^1_{\mathrm{per},0}(Y)$, $1\leq j\leq n$, to be the unique solution in $H^1_{\mathrm{per},0}(Y)$ to the problem
\begin{align}\label{phi_j prob}
-\Delta \phi_j = \nabla\cdot (rA e_j)\quad\text{in }Y,\qquad \phi_j\text{ is $Y$-periodic},\qquad \int_Y \phi_j = 0,
\end{align}
where $e_j\in \R^n$ denotes the $j$-th column of $I_n$. We set $\phi:=(\phi_1,\dots,\phi_n)\in H^1_{\mathrm{per},0}(Y;\R^n)$.
\begin{rmrk}\label{Rk: Prop of phi:j}
We make the following observations:
\begin{itemize}
\item[(i)] For any $L\in \N$ and $1\leq j\leq n$, the function $\phi_{j,L}:=\phi_j$ is the unique solution in $H^1_{\mathrm{per},0}(Y_L)$ to 
\begin{align}\label{phi_j,L}
-\Delta \phi_{j,L} = \nabla\cdot (rA e_j)\quad\text{in }Y_L,\qquad \phi_{j,L}\text{ is $Y_L$-periodic},\qquad \int_{Y_L} \phi_{j,L} = 0,
\end{align}
where $Y_L:=(0,L)^n$. Indeed, note $\phi_{j,L} = \phi_{j,L}(\cdot + k)$ for any $k\in \Z^n$ by uniqueness of solutions to \eqref{phi_j,L}. Hence, $\phi_{j,L}\in H^1_{\mathrm{per},0}(Y)$ solves \eqref{phi_j prob} and thus, $\phi_{j,L} = \phi_j$.
\item[(ii)] There holds $\nabla\cdot \phi = 0$ almost everywhere. Indeed, this follows from $\nabla\cdot \phi\in L^2_{\mathrm{per},0}(Y)$ and 
\begin{align*}
(\nabla\cdot \phi,\Delta\fhi)_{L^2(Y)}  = \sum_{j=1}^n (\nabla \phi_j,\nabla (\partial_j\fhi))_{L^2(Y)} = -\sum_{j=1}^n (rAe_j,\nabla (\partial_j\fhi))_{L^2(Y)}  = 0
\end{align*}
for any $\fhi\in H^2_{\mathrm{per}}(Y)$, where we have used in the last step that $(r,-A:D^2 \fhi)_{L^2(Y)} = 0$.
\end{itemize}
\end{rmrk}

Next, we define the skew-symmetric matrix-valued map $B = (b_{ij})_{1\leq i,j\leq n}\in L^2_{\mathrm{per},0}(Y;\R^{n\times n})$ by $b_{ij}:=\partial_i \phi_j - \partial_j \phi_i$ for $1\leq i,j\leq n$ and we set
\begin{align*}
M:=rA+B \in L^2_{\mathrm{per}}(Y;\R^{n\times n}).
\end{align*} 
We observe that for any $L\in \N$ and $1\leq j\leq n$, writing $Y_L:=(0,L)^n$, we have that
\begin{align*}
-(Me_j,\nabla \fhi)_{L^2(Y_L)} =(\nabla \phi_j - Be_j, \nabla \fhi)_{L^2(Y_L)} =  (\partial_j\phi,\nabla \fhi)_{L^2(Y_L)} = (\nabla\cdot \phi, \partial_j \fhi)_{L^2(Y_L)} = 0
\end{align*}
for any $\fhi\in H^1_{\mathrm{per}}(Y_L)$, where we have used Remark \ref{Rk: Prop of phi:j}(i) in the first equality and Remark \ref{Rk: Prop of phi:j}(ii) in the last equality. Hence, we have for any bounded domain $\omega\subset \R^n$ that 
\begin{align}\label{div M = 0 weakly}
(Me_j,\nabla w)_{L^2(\omega)} = 0\qquad \forall 1\leq j\leq n\quad \forall w\in H^1_0(\omega).
\end{align}
Indeed, let $L\in \N$ such that $\omega\subset k_L+Y_L$ with $k_L:=-\frac{L}{2}(1,\dots,1)\in \R^n$, extend $w\in H^1_0(\Omega)$ to a function $\tilde{w}\in H^1_0(k_L+Y_L)$ by setting $\tilde{w} = 0$ a.e. in $(k_L+Y_L)\backslash \omega$, and define $\fhi\in H^1_{\mathrm{per}}(Y_L)$ to be the $Y_L$-periodic extension of $\tilde{w}$ to see that $(Me_j,\nabla w)_{L^2(\omega)} = (Me_j,\nabla \tilde{w})_{L^2(k_L+Y_L)} = (Me_j,\nabla \fhi)_{L^2(Y_L)} = 0$.

Writing $M^{\eps}:=M(\frac{\cdot}{\eps})$, we then obtain that
\begin{align}\label{divform eqn}
(M^{\eps}\nabla u_{\eps},\nabla v)_{L^2(\Omega)} = (-M^{\eps}:D^2 u_{\eps}, v)_{L^2(\Omega)} = (r^{\eps}f,v)_{L^2(\Omega)}\qquad \forall v\in C_c^{\infty}(\Omega),
\end{align}  
where the first equality follows from \eqref{div M = 0 weakly} and the second equality follows from \eqref{weighted with r} and skew-symmetry of $B$. Finally, noting that due to the fact that $r\in L^2_{\mathrm{per}}(Y)$, $M\in L^2_{\mathrm{per}}(Y;\R^{n\times n})$ and $\int_Y B = 0$ we have 
\begin{align}\label{lim of M}
M^{\eps} &\weak \int_Y M = \int_Y rA\;\text{ weakly in }L^2(\Omega;\R^{n\times n}),\;\;
r^{\eps} \weak \int_Y r = 1\;\text{ weakly in }L^2(\Omega)\;\text{ as }\eps\searrow 0,
\end{align}
we can use \eqref{limit for subs} to pass to the limit $\eps\searrow 0$ in \eqref{divform eqn} to obtain the following convergence result:

\begin{thrm}[convergence result]\label{Thm: Conv}
Let $\Omega\subset \R^n$ be a bounded convex domain, $A\in \mathcal{M}(\lambda,\Lambda)$ for some $\lambda,\Lambda > 0$, $f\in L^2(\Omega)$, $g\in H^2(\Omega)$, and suppose that \eqref{Cordes} holds. Denoting the invariant measure to $A$ by $r\in L^2_{\mathrm{per}}(Y)$, we introduce $\bar{A}:= \int_Y rA\in \R^{n\times n}_{\mathrm{sym}}$. Then, there holds $\lambda \lvert \xi\rvert^2\leq \bar{A}\xi\cdot \xi\leq \Lambda \lvert \xi\rvert^2$ for all $\xi\in \R^n$, and the sequence of solutions $(u_{\eps})_{\eps > 0}\subset H^2(\Omega)$ to \eqref{ueps problem again} converges weakly in $H^2(\Omega)$ to the unique solution $u\in H^2(\Omega)$ to the homogenized problem
\begin{align}\label{u problem}
\begin{split}
\bar{L}u:=-\bar{A}:D^2 u &= f\quad\text{in }\Omega,\\ u &= g\quad\text{on }\partial\Omega.
\end{split}
\end{align} 
We call the symmetric positive definite matrix $\bar{A}$ the effective coefficient matrix.
\end{thrm}

\subsection{Corrector estimates}\label{Sec: corrector est}

It is by now standard how to obtain corrector estimates to any order in the classical $A\in C^{0,\alpha}$ setting, introducing suitable interior and boundary correctors and assuming that $u$ is sufficiently regular; see e.g. \cite{KL16,ST21}. For $A\in L^{\infty}$, using the uniform $H^2$-bound from Theorem \ref{Thm: Well-posedness of uepsproblem} and using Theorem \ref{Thm: per nondiv}, one obtains the following results by arguing along the lines of \cite{ST21}. To simplify the statement of the results, we introduce the following notation: 

\begin{dfntn}[the notation $T(A,a)$]\label{Def: T(A,a)}
Let $a\in L^2_{\mathrm{per}}(Y)$, $A\in \mathcal{M}(\lambda,\Lambda)$ for some $\lambda,\Lambda > 0$, and suppose that $A$ satisfies \eqref{Cordes}. Let $r\in L^2_{\mathrm{per}}(Y)$ denote the invariant measure to $A$. Then, we define $T(A,a):=w$ to be the unique element $w\in H^2_{\mathrm{per},0}(Y)$ satisfying $-A:D^2 w = a-(a,r)_{L^2(Y)}$.
\end{dfntn}

\begin{thrm}[an $\calO(\eps)$ corrector estimate in the $H^2$-norm]\label{Thm: corrI}
Suppose that we are in the situation of Theorem \ref{Thm: Conv} and assume that $u\in H^4(\Omega)$. Set $v_{ij}:=T(A,a_{ij})\in H^2_{\mathrm{per},0}(Y)$ for $1\leq i,j\leq n$ and set $V := (v_{ij})_{1\leq i,j\leq n}$. Writing $\fhi^{\eps}:=\fhi(\frac{\cdot}{\eps})$ for $\fhi\in\{V,\partial_i V\}$, we define $\eta_{\eps}:=V^{\eps}:D^2 u$ and 
\begin{align*}
p_{ij,\eps}:= 2 [\partial_i V]^{\eps} :D^2(\partial_{j}u)  + \eps V^{\eps}:D^2(\partial^2_{ij}u),\qquad P_{\eps}:= (p_{ij,\eps})_{1\leq i,j\leq n}.
\end{align*}
Then, writing $C_{\delta}:=(1-\sqrt{1-\delta})^{-1}>0$ and assuming that
\begin{align}\label{ass corr1}
\eta_{\eps}\in H^2(\Omega)\quad\text{and}\quad P_{\eps}\in L^2(\Omega)\text{ with } \|P_{\eps}\|_{L^2(\Omega)} = \calO(1),
\end{align}
there exists a constant $c_0 = c_0(\mathrm{diam}(\Omega),n)>0$ such that
\begin{align}\label{basic corrector bound}
\left\|u_{\eps}-u-\eps^2 (\eta_{\eps} + \theta_{\eps})\right\|_{H^2(\Omega)} \leq \frac{\Lambda}{\lambda} C_{\delta}c_0 \|P_{\eps}\|_{L^2(\Omega)} \,\eps = \calO(\eps),
\end{align} 
where $\theta_{\eps}\in H^2(\Omega)$ denotes the unique solution to $L_{\eps} \theta_{\eps} = 0$ in $\Omega$, $\theta_{\eps} = -\eta_{\eps}$ on $\partial\Omega$.
\end{thrm}

\begin{rmrk}\label{Rk: const c0}
The proof of Theorem \ref{Thm: corrI} shows that \eqref{basic corrector bound} holds with
\begin{align*}
c_0:=\sqrt{n}\,\sup_{v\in H^2(\Omega)\cap H^1_0(\Omega)\backslash\{0\}} \frac{\|v\|_{H^2(\Omega)}}{\|\Delta v\|_{L^2(\Omega)}}.
\end{align*}
\end{rmrk}

\begin{rmrk}\label{Rk: CorrI in 123d}
If $n\leq 3$, in view of the Sobolev embeddings, $u\in H^4(\Omega)$ is sufficient to guarantee that the assumption \eqref{ass corr1} in Theorem \ref{Thm: corrI} is met.
\end{rmrk}

\begin{cor}[$L^{\infty}$-rate in dimension $n\leq 3$]\label{Cor: Linfty rate}
Suppose that we are in the situation of Theorem \ref{Thm: Conv} and $n\leq 3$. Then, if $u\in H^4(\Omega)$, we have that $\|u_{\eps}-u\|_{L^{\infty}(\Omega)} = \calO(\eps)$.
\end{cor}
Note that the result of Corollary \ref{Cor: Linfty rate} follows directly from Theorem \ref{Thm: corrI}, Remark \ref{Rk: CorrI in 123d}, the Sobolev embedding $H^2(\Omega)\hookrightarrow L^{\infty}(\Omega)$, and the fact that by the maximum principle (see e.g., \cite{Saf10}) we have the bound $\|\theta_{\eps}\|_{L^{\infty}(\Omega)}\leq \|V^{\eps}:D^2 u\|_{L^{\infty}(\Omega)}\leq \|V\|_{L^{\infty}(\R^n)}\|D^2 u\|_{L^{\infty}(\Omega)}$ in dimension $n\leq 3$.

Similarly, we can obtain an $\calO(\eps^2)$ corrector estimate in the $H^2(\Omega)$-norm.

\begin{thrm}[an $\calO(\eps^2)$ corrector estimate in the $H^2$-norm]\label{Thm: CorrII}
Suppose that we are in the situation of Theorem \ref{Thm: corrI} and assume that $u\in H^5(\Omega)$. Set $\chi_{jkl}:=T(A,Ae_j\cdot \nabla v_{kl})\in H^2_{\mathrm{per},0}(Y)$ for $1\leq j,k,l\leq n$ and set $X := (\chi_{jkl})_{1\leq j,k,l\leq n}$. Writing $\fhi^{\eps}:=\fhi(\frac{\cdot}{\eps})$ for $\fhi\in\{V,X,\partial_i X\}$, we define $\tilde{\eta}_{\eps}:=X^{\eps}:D^3 u$ and 
\begin{align*}
q_{ij,\eps}:= V^{\eps}:D^2(\partial^2_{ij} u) + 4 [\partial_i X]^{\eps} :D^3(\partial_{j}u)  + 2\eps X^{\eps}:D^3(\partial^2_{ij}u),\qquad Q_{\eps}:= (q_{ij,\eps})_{1\leq i,j\leq n}.
\end{align*}
Then, writing $C_{\delta}:=(1-\sqrt{1-\delta})^{-1}$ and assuming that
\begin{align}\label{ass corr2}
\tilde{\eta}_{\eps}\in H^2(\Omega)\quad\text{and}\quad Q_{\eps}\in L^2(\Omega)\text{ with } \|Q_{\eps}\|_{L^2(\Omega)} = \calO(1),
\end{align}
we have the bound
\begin{align}\label{corrector bound 2}
\|u_{\eps}-u+2\eps z_{\eps}-\eps^2 (\eta_{\eps}+\theta_{\eps})-2\eps^3(\tilde{\eta}_{\eps}+\tilde{\theta}_{\eps})\|_{H^2(\Omega)} \leq \frac{\Lambda}{\lambda} C_{\delta}c_0 \|Q_{\eps}\|_{L^2(\Omega)}\,\eps^2 = \calO(\eps^2),
\end{align} 
where $c_0>0$ is as in Remark \ref{Rk: const c0} and $\tilde{\theta}_{\eps},z_{\eps}\in H^2(\Omega)$ denote the unique solutions to $L_{\eps} \tilde{\theta}_{\eps} = 0$ in $\Omega$, $\tilde{\theta}_{\eps} = -\tilde{\eta}_{\eps}$ on $\partial\Omega$, and $L_{\eps}z_{\eps}= -\sum_{j,k,l=1}^n (Ae_j\cdot \nabla v_{kl},r)_{L^2(Y)}\partial^3_{jkl} u$ in $\Omega$, $z_{\eps} = 0$ on $\partial\Omega$, respectively.
\end{thrm}

\begin{rmrk}\label{Rk: CorrII in 123d}
If $n\leq 3$, in view of the Sobolev embeddings, $u\in H^5(\Omega)$ is sufficient to guarantee that the assumption \eqref{ass corr2} in Theorem \ref{Thm: CorrII} is met.
\end{rmrk}

\begin{cor}[$L^{\infty}$-bound in dimension $n\leq 3$]\label{Cor: Linfty rate II}
Suppose that we are in the situation of Theorem \ref{Thm: Conv} and $n\leq 3$. Then, if $u\in H^5(\Omega)$, we have that 
\begin{align}\label{Linfty bd 2}
\|u_{\eps}-u+2\eps z_{\eps}\|_{L^{\infty}(\Omega)} = \calO(\eps^2),
\end{align}
where $z_{\eps}\in H^2(\Omega)$ is defined as in Theorem \ref{Thm: CorrII}.
\end{cor}

Similarly to Theorems \ref{Thm: corrI} and \ref{Thm: CorrII}, corrector estimates in the $H^2(\Omega)$-norm can be obtained to any order, assuming that $u$ is sufficiently regular and constructing suitable corrector functions.

\subsection{Type-$\eps^2$ diffusion matrices}\label{Sec: Type}

First, let us generalize the definition of type-$\eps^2$ diffusion matrices from \cite{GST22} given for $A\in C^{0,\alpha}$ to our present situation.

\begin{dfntn}[type-$\eps^2$ and type-$\eps$ diffusion matrices]\label{Def: type}
Let $A\in \mathcal{M}(\lambda,\Lambda)$ for some $\lambda,\Lambda > 0$. We set $c_j^{kl}(A):=(Ae_j\cdot \nabla v_{kl},r)_{L^2(Y)}$ for $1\leq j,k,l\leq n$, where $r\in L^2_{\mathrm{per}}(Y)$ denotes the invariant measure to $A$ and $v_{kl}:=T(A,a_{kl})\in H^2_{\mathrm{per},0}(Y)$. We call $A$ a type-$\eps^2$ diffusion matrix if $C_{jkl}(A):=c_j^{kl}(A) + c_k^{jl}(A) + c_l^{jk}(A) = 0$ for all $1\leq j,k,l\leq n$. Otherwise, we call $A$ a type-$\eps$ diffusion matrix.
\end{dfntn}

Note that the function $z_{\eps}$ in \eqref{corrector bound 2} and \eqref{Linfty bd 2} vanishes for any $u\in H^5(\Omega)$ if and only if the symmetric part of the third-order homogenized tensor $(c_j^{kl})_{1\leq j,k,l\leq n}$ vanishes, or equivalently, if and only if $A$ is type-$\eps^2$. In particular, in the situation of Corollary \ref{Cor: Linfty rate II}, if $A$ is type-$\eps^2$ we have that $\|u_{\eps}-u\|_{L^{\infty}(\Omega)} = \calO(\eps^2)$  whereas if $A$ is type-$\eps$ we only have that $\|u_{\eps}-u\|_{L^{\infty}(\Omega)} = \calO(\eps)$ and this rate $\calO(\eps)$ is optimal in general.

We can now extend the main results from \cite{GST22} obtained for $A\in C^{0,\alpha}$ to our present case.  The following results can then be proved by using arguments almost identical to \cite{GST22}.

\begin{thrm}[diffusion matrices $A = C+aM$ are type-$\eps^2$]\label{Thm: C+aM}
Let $C,M\in \R^{n\times n}_{\mathrm{sym}}$, let $a\in L^{\infty}_{\mathrm{per}}(Y)$, and suppose that $A := C+aM$ is such that $A\in \mathcal{M}(\lambda,\Lambda)$ for some $\lambda,\Lambda>0$ and \eqref{Cordes} holds. Set $w:=T(A,a)\in H^2_{\mathrm{per},0}(Y)$ and $v_{ij}:=T(A,a_{ij})\in H^2_{\mathrm{per},0}(Y)$ for $1\leq i,j\leq n$. Then, we have the following results:
\begin{itemize}
\item[(i)] The invariant measure to $A$ is given by $r = 1+M:D^2 w$.
\item[(ii)] We have that $V:=(v_{ij})_{1\leq i,j\leq n}$ is given by $V = wM$.
\item[(iii)] There holds $-C:D^2 w = ra-\bar{a}$ where $\bar{a}:=(a,r)_{L^2(Y)}$.
\item[(iv)] The third-order homogenized tensor vanishes, i.e., $c_j^{kl}(A) = 0$ for all $1\leq j,k,l\leq n$.
\end{itemize}
In particular, $A$ is a type-$\eps^2$ diffusion matrix.
\end{thrm}

\begin{thrm}[characterization of type-$\eps^2$ diagonal diffusion matrices for $n=2$]\label{Thm: type-eps^2 diag}
Let $A\in \mathcal{M}(\lambda,\Lambda)$ for $n=2$ and some $\lambda,\Lambda > 0$, and suppose that $\frac{1}{a}A = I_2 + bM =:B$ for some $b\in L^{\infty}_{\mathrm{per}}(Y)$, where $a:=\frac{1}{2}\mathrm{tr}(A)$ and $M:=\mathrm{diag}(1,-1)$. Set $w_A:=T(A,a)\in H^2_{\mathrm{per},0}(Y)$, $w_B:=T(B,b)\in H^2_{\mathrm{per},0}(Y)$, and $v_{ij}:=T(A,a_{ij})\in H^2_{\mathrm{per},0}(Y)$ for $1\leq i,j\leq n$. Then, we have the following results:
\begin{itemize}
\item[(i)] The invariant measure $r_B$ to $B$ is given by $r_B = 1+M:D^2 w_B$, and the invariant measure $r_A$ to $A$ is given by $r_A = \frac{\bar{a}}{a}r_B$ where $\bar{a}:= (\frac{1}{a},r_B)_{L^2(Y)}^{-1} = (a,r_A)_{L^2(Y)}$.
\item[(ii)] We have that $V:=(v_{ij})_{1\leq i,j\leq n}$ is given by $V = w_A(I_2 + \bar{b}M)+w_B M$ where  $\bar{b}:=(b,r_B)_{L^2(Y)}$.
\item[(iii)] There holds $-\Delta w_B = r_B b - \bar{b}$.
\item[(iv)] There holds $c^{kl}_{3-s}(A) = 2(-1)^{s+1} \bar{a}(1+m_{kl}\bar{b})(\partial_{3-s} w_A,\partial^2_{ss} w_B)_{L^2(Y)}$ for any $s,k,l\in\{1,2\}$.
\end{itemize}
In particular, $A$ is type-$\eps^2$ if, and only if, $(\partial_1 w_A,\partial^2_{22} w_B)_{L^2(Y)} = (\partial_2 w_A,\partial^2_{11} w_B)_{L^2(Y)} = 0$.
\end{thrm}

Note that any diagonal $A = \mathrm{diag}(a_{11},a_{22}) \in \mathcal{M}(\lambda,\Lambda)$ for $n=2$ can be written in the form specified in Theorem \ref{Thm: type-eps^2 diag} by setting $b:=\frac{a_{11}-a_{22}}{a_{11}+a_{22}}$. 

\begin{rmrk}
As a consequence of Theorem \ref{Thm: type-eps^2 diag} we have that any diagonal constant-trace diffusion matrix $A\in \mathcal{M}(\lambda,\Lambda)$ for $n=2$ and some $\lambda,\Lambda > 0$ is type-$\eps^2$.
\end{rmrk}

\section{Numerical Methods}\label{Sec: Num methods}

\subsection{Finite element approximation of the invariant measure}

Let $A\in \mathcal{M}(\lambda,\Lambda)$ for some $\lambda,\Lambda > 0$ and suppose that \eqref{Cordes} holds. In this section, we discuss the finite element approximation of the invariant measure $r\in L^2_{\mathrm{per}}(Y)$ to $A$, i.e., the unique solution $r\in L^2_{\mathrm{per}}(Y)$ to
\begin{align*}
-D^2:(rA) = 0\quad\text{in }Y,\qquad r\text{ is Y-periodic},\qquad \int_Y r = 1.
\end{align*} 
We refer to Section \ref{Sec: Per problems} for the existence, uniqueness, and $L^2$-bounds for the invariant measure. We can construct simple finite element methods (FEMs) for the numerical solution of this problem based on our observation from Theorem \ref{Thm: invmeas} that
\begin{align}\label{num r motivation}
r = c^{-1} \gamma \tilde{r},\quad\text{where}\quad \tilde{r} := 1-\Delta \psi,\quad c:=(\gamma,\tilde{r})_{L^2(Y)},
\end{align}
and $\psi\in H^2_{\mathrm{per},0}(Y)$ is the unique element in $H^2_{\mathrm{per},0}(Y)$ such that
\begin{align}\label{psi problem}
b_0(\fhi,\psi) = (\gamma,A:D^2 \fhi)_{L^2(Y)}\qquad \forall \fhi\in H^2_{\mathrm{per},0}(Y),
\end{align}
where $\gamma\in L^{\infty}_{\mathrm{per}}(Y)$ is defined as in Remark \ref{Rk: Cordes holds} and $b_0:H^2_{\mathrm{per}}(Y)\times H^2_{\mathrm{per}}(Y)\rightarrow \R$ is defined as in Lemma \ref{Lmm: b_mu}. Let us recall that $\tilde{r}\in L^2_{\mathrm{per}}(Y)$ is the unique invariant measure to $\tilde{A}:=\gamma A\in \mathcal{M}(\frac{\lambda^2}{\Lambda^2},\frac{\Lambda^2}{\lambda^2})$.

\subsubsection{Approximation of $\tilde{r}$ via $H^2_{\mathrm{per},0}(Y)$-conforming FEM for $\psi$}

Noting that $b_0$ defines a coercive (Lemma \ref{Lmm: b_mu}) bounded bilinear form on $H^2_{\mathrm{per},0}(Y)$ and $\fhi\mapsto (\gamma,A:D^2 \fhi)_{L^2(Y)}$ defines a bounded linear functional on $H^2_{\mathrm{per},0}(Y)$, we obtain the following result by the Lax--Milgram theorem and standard conforming finite element theory:

\begin{lmm}[approximation of $\tilde{r}$ via $H^2_{\mathrm{per},0}(Y)$-conforming FEM for $\psi$]\label{Lmm: H2 approx rtilde}
Let $A\in \mathcal{M}(\lambda,\Lambda)$ for some $\lambda,\Lambda > 0$ and suppose that \eqref{Cordes} holds. Let $\gamma\in L^{\infty}_{\mathrm{per}}(Y)$ be defined as in Remark \ref{Rk: Cordes holds} and let $\tilde{r}\in L^2_{\mathrm{per}}(Y)$ denote the invariant measure to $\tilde{A}:=\gamma A$. Let $b_0:H^2_{\mathrm{per}}(Y)\times H^2_{\mathrm{per}}(Y)\rightarrow \R$ and $C_{\delta}>0$ be defined as in Lemma \ref{Lmm: b_mu}, and let $\Psi_h\subset H^2_{\mathrm{per},0}(Y)$ be a closed linear subspace of $H^2_{\mathrm{per},0}(Y)$. Then, there exists a unique $\psi_h\in \Psi_h$ such that
\begin{align}\label{psi h prob}
b_0(\fhi_h,\psi_h)  = (\gamma,A:D^2 \fhi_h)_{L^2(Y)}\qquad \forall \fhi_h\in \Psi_h.
\end{align} 
Further, setting $\tilde{r}_h:=1-\Delta \psi_h \in L^2_{\mathrm{per}}(Y)$ we have that
\begin{align*}
\|\tilde{r}-\tilde{r}_h\|_{L^2(Y)} = \|\Delta (\psi - \psi_h)\|_{L^2(Y)} \leq \sqrt{n}\frac{\Lambda}{\lambda}C_{\delta} \inf_{\fhi_h\in \Psi_h}\|\Delta(\psi-\fhi_h)\|_{L^2(Y)},
\end{align*}
where $\psi\in H^2_{\mathrm{per},0}(Y)$ denotes the unique element in $H^2_{\mathrm{per},0}(Y)$ satisfying \eqref{psi problem}.
\end{lmm}

In practice, $H^2$-conforming elements such as the the Argyris or the HCT element need to be used on a periodic shape-regular triangulation of the unit cell. A more attractive alternative from an implementation point of view is the approach presented next.

\subsubsection{Approximation of $\tilde{r}$ via $H^1_{\mathrm{per},0}(Y;\R^n)$-conforming FEM}

As an alternative to the $H^2_{\mathrm{per},0}(Y)$-conforming finite element method, we propose an approximation scheme for $\tilde{r}$ based on an $H^1_{\mathrm{per},0}(Y;\R^n)$-conforming FEM, using ideas from  \cite{Gal17b}. To this end, let us introduce the bilinear form $b:H^1_{\mathrm{per}}(Y;\R^n)\times H^1_{\mathrm{per}}(Y;\R^n)\rightarrow \R$ given by
\begin{align}\label{bil form a}
b(w,\tilde{w}):=(\gamma A:Dw,\nabla\cdot \tilde{w})_{L^2(Y)}+S(w,\tilde{w}),\quad   S(w,\tilde{w}):=\frac{1}{2}(Dw - Dw^{\mathrm{T}},D\tilde{w} - D\tilde{w}^{\mathrm{T}})_{L^2(Y)}
\end{align}
for $w,\tilde{w}\in H^1_{\mathrm{per}}(Y;\R^n)$, where $\gamma\in L^{\infty}_{\mathrm{per}}(Y)$ is defined as in Remark \ref{Rk: Cordes holds}. Note that the stabilization term can be equivalently written as $S(w,\tilde{w}) = (\nabla \times w,\nabla \times \tilde{w})_{L^2(Y)}$ if $w,\tilde{w}\in H^1_{\mathrm{per}}(Y;\R^3)$, and as $S((w_1,w_2),(\tilde{w}_1,\tilde{w}_2)) = (\partial_2 w_1-\partial_1 w_2,\partial_2 \tilde{w}_1-\partial_1 \tilde{w}_2)_{L^2(Y)}$ if $(w_1,w_2),(\tilde{w}_1,\tilde{w}_2)\in H^1_{\mathrm{per}}(Y;\R^2)$. We can show the following result:

\begin{lmm}[approximation of $\tilde{r}$ via $H^1_{\mathrm{per},0}(Y;\R^n)$-conforming FEM]\label{Lmm: Approx rtilde mixed}
Let $A\in \mathcal{M}(\lambda,\Lambda)$ for some $\lambda,\Lambda > 0$ and suppose that \eqref{Cordes} holds. Let $\gamma\in L^{\infty}_{\mathrm{per}}(Y)$ be defined as in Remark \ref{Rk: Cordes holds} and let $\tilde{r}\in L^2_{\mathrm{per}}(Y)$ denote the invariant measure to $\tilde{A}:=\gamma A$. Let $b:H^1_{\mathrm{per}}(Y;\R^n)\times H^1_{\mathrm{per}}(Y;\R^n)\rightarrow \R$ be defined by \eqref{bil form a}, $C_{\delta}>0$ be defined as in Lemma \ref{Lmm: b_mu}, and let $P_h\subset H^1_{\mathrm{per},0}(Y;\R^n)$ be a closed linear subspace of $H^1_{\mathrm{per},0}(Y;\R^n)$. Then, the following assertions hold.
\begin{itemize}
\item[(i)] There exists a unique $p\in H^1_{\mathrm{per},0}(Y;\R^n)$ such that $b(w,p)=(\gamma,A:Dw)_{L^2(Y)}$ for all $w\in H^1_{\mathrm{per},0}(Y;\R^n)$. Further, there holds $\tilde{r} = 1 - \nabla\cdot p$.
\item[(ii)] There exists a unique $p_h\in P_h$ such that $b(w_h,p_h)=(\gamma,A:Dw_h)_{L^2(Y)}$ for all $w_h\in P_h$.
\item[(iii)] Setting $\tilde{r}_h:=1-\nabla\cdot p_h\in L^2_{\mathrm{per}}(Y)$ we have that
\begin{align*}
\|\tilde{r}-\tilde{r}_h\|_{L^2(Y)} = \|\nabla\cdot (p-p_h)\|_{L^2(Y)}\leq \left(1+\sqrt{n}\frac{\Lambda}{\lambda}\right)C_{\delta} \inf_{w_h\in P_h}\|D(p-w_h)\|_{L^2(Y)},
\end{align*}
where $p,p_h\in H^1_{\mathrm{per},0}(Y;\R^n)$ are the functions from (i)--(ii).
\end{itemize}
\end{lmm}

\subsubsection{Approximation of $r$}

In view of Lemma \ref{Lmm: H2 approx rtilde} and Lemma \ref{Lmm: Approx rtilde mixed} we are able to obtain a sequence of approximations $(\tilde{r}_h)_{h}\subset L^2_{\mathrm{per}}(Y)$ to $\tilde{r}$, i.e., $\|\tilde{r}-\tilde{r}_h\|_{L^2(Y)}\rightarrow 0$ as $h\searrow 0$, by choosing a suitable finite element space $\Psi_h\subset H^2_{\mathrm{per},0}(Y)$ (for the method from Lemma \ref{Lmm: H2 approx rtilde}) or $P_h\subset H^1_{\mathrm{per},0}(Y;\R^n)$ (for the method from Lemma \ref{Lmm: Approx rtilde mixed}) corresponding to a shape-regular periodic triangulation $\calT_h$ with mesh-size $h > 0$. It is also standard that the results from Lemma \ref{Lmm: H2 approx rtilde} and \ref{Lmm: Approx rtilde mixed} imply convergence rates depending on the regularity of $\psi$ by using interpolation estimates. Then, in view of \eqref{num r motivation}, we can obtain an approximation to $r$.

\begin{thrm}[approximation of $r$]\label{Thm: approx of r}
Let $A\in \mathcal{M}(\lambda,\Lambda)$ for some $\lambda,\Lambda > 0$ and suppose that \eqref{Cordes} holds. Let $r\in L^2_{\mathrm{per}}(Y)$ denote the invariant measure to $A$ and let $\tilde{r}\in L^2_{\mathrm{per}}(Y)$ denote the invariant measure to $\tilde{A}:=\gamma A$, where $\gamma\in L^{\infty}_{\mathrm{per}}(Y)$ is defined as in Remark \ref{Rk: Cordes holds}. Let $(\tilde{r}_h)_{h>0}\subset L^2_{\mathrm{per}}(Y)$ with $\|\tilde{r}-\tilde{r}_h\|_{L^2(Y)}\rightarrow 0$ as $h\searrow 0$. Then, for $h> 0$ sufficiently small we have that 
\begin{align}\label{rh defn}
r_h:=c_h^{-1} \gamma\tilde{r}_h\in L^2_{\mathrm{per}}(Y),\qquad c_h:= (\gamma,\tilde{r}_h)_{L^2(Y)}
\end{align}
is well-defined and satisfies the bound
\begin{align*}
\|r-r_h\|_{L^2(Y)}\leq \frac{\Lambda}{\lambda^2}\left(1+\frac{\Lambda^2}{\lambda}\right)\left(1+\frac{\Lambda}{\lambda^2}\|r\|_{L^2(Y)}\right)\|\tilde{r}-\tilde{r}_h\|_{L^2(Y)}.
\end{align*}
\end{thrm}

\subsection{Approximation of the homogenized problem}

In view of Lemma \ref{Lmm: H2 approx rtilde}, Lemma \ref{Lmm: Approx rtilde mixed}, and Theorem \ref{Thm: approx of r}, we now know how to obtain a sequence of approximations $(r_h)_h\subset L^2_{\mathrm{per}}(Y)$ to  $r$ with $\|r-r_h\|_{L^2(Y)}\rightarrow 0$. Recalling that the effective coefficient matrix $\bar{A}\in \R^{n\times n}_{\mathrm{sym}}$ is defined as $\bar{A}:=\int_Y rA$, we introduce the approximate effective coefficient matrix $\bar{A}_h\in \R^{n\times n}_{\mathrm{sym}}$ by
\begin{align}\label{Abar h def}
\bar{A}_{h}:=\int_Y r_h A.
\end{align}
We have the following error bound for this approximation of the effective coefficient matrix:
\begin{lmm}[approximation of $\bar{A}$]\label{Lmm: Approx Abar}
Let $A\in \mathcal{M}(\lambda,\Lambda)$ for some $\lambda,\Lambda > 0$ and suppose that \eqref{Cordes} holds. Let $r\in L^2_{\mathrm{per}}(Y)$ denote the invariant measure to $A$ and let $(r_h)_{h>0}\subset L^2_{\mathrm{per}}(Y)$ be such that $\|r-r_h\|_{L^2(Y)}\rightarrow 0$ as $h\searrow 0$. Let $\bar{A}\in \R^{n\times n}_{\mathrm{sym}}$ denote the effective coefficient matrix corresponding to $A$, and let $\bar{A}_h\in \R^{n\times n}_{\mathrm{sym}}$ be defined by \eqref{Abar h def}. Then, we have that
\begin{align*}
\lvert \bar{A} - \bar{A}_{h}\rvert \leq \sqrt{n}\Lambda \|r-r_h\|_{L^1(Y)}\leq \sqrt{n}\Lambda \|r-r_h\|_{L^2(Y)}
\end{align*}
and $\bar{A}_h$ is positive definite for $h>0$ sufficiently small.
\end{lmm}

With the approximate effective coefficient matrix at hand, we can obtain an approximation to the solution of the homogenized problem:
\begin{lmm}[approximation of $u$]\label{Lmm: Approx u}
Suppose that we are in the situation of Lemma \ref{Lmm: Approx Abar}. Let $\Omega\subset \R^n$ be a bounded convex domain, $f\in L^2(\Omega)$, $g\in H^2(\Omega)$, and let $u\in H^2(\Omega)$ be the solution to the homogenized problem \eqref{u problem}. Then, for $h>0$ sufficiently small, there exists a unique solution $u_h\in H^2(\Omega)$ to 
\begin{align}\label{uh prob}
\begin{split}
\bar{L}_h u_h:=-\bar{A}_h:D^2 u_h &= f\quad\text{in }\Omega,\\
u_h &= g\quad \text{on }\partial\Omega,
\end{split}
\end{align}
and we have the bound 
\begin{align*}
\|u-u_h\|_{H^2(\Omega)}\leq C\lvert \bar{A}-\bar{A}_h\rvert \left(\|f\|_{L^2(\Omega)} + \|g\|_{H^2(\Omega)}\right)
\end{align*}
for some constant $C = C(\mathrm{diam}(\Omega),\lambda,\Lambda,n)>0$.
\end{lmm}

Let us note that, since $\bar{A}_h$ is a constant symmetric positive definite matrix for $h>0$ sufficiently small, the solution to \eqref{uh prob} can be approximated by a standard $H^1(\Omega)$-conforming finite element method. If an $H^2$-approximation of $u_h$ is desired, the function $\tilde{u}_h:=u_h-g\in H^2(\Omega)\cap H^1_0(\Omega)$ can be approximated by an $H^2(\Omega)$-conforming finite element method based on the variational formulation $(\bar{L}_h\tilde{u}_h,\bar{L}_h v)_{L^2(\Omega)} = (f-\bar{L}_h g,\bar{L}_h v)_{L^2(\Omega)}$ for all $v\in H^2(\Omega)\cap H^1_0(\Omega)$.

\section{Numerical Experiments}\label{Sec: Exp}

In this section, we perform numerical experiments illustrating the theoretical results from Section \ref{Sec: Num methods}. We take $n=2$ and choose $A\in L^{\infty}(\R^2;\R^{2\times 2}_{\mathrm{sym}})$ as
\begin{align}\label{A for numexp}
A(y):=\mathrm{diag}(1-a(y),a(y)),\qquad a(y_1,y_2):= \frac{3-2\omega (y_1)\sin(2\pi y_2)}{8+(\pi^2 \theta(y_1)-2\omega (y_1))\sin(2\pi y_2)}
\end{align}
for $y=(y_1,y_2)\in\R^2$, where 
\begin{align*}
\omega(t):=\mathrm{sign}(\sin(2\pi t)),\qquad \theta(t):=S(t)(1-\omega(t)S(t)),\qquad S(t):=1-2(t-\floor{t})
\end{align*}
for $t\in \R$. Note that $A\in \mathcal{M}(\frac{1}{10},\frac{9}{10})$ since $a\in L^{\infty}_{\mathrm{per}}(Y)$ and $\frac{1}{10}\leq a\leq \frac{5}{6}$ a.e. in $\R^2$ (note $\lvert \pi^2 \theta - 2\omega\rvert \leq 2$ a.e. in $\R$). In particular, in view of Remark \ref{Rk: Cordes in 1D,2D}, $A$ satisfies the Cordes condition \eqref{Cordes} with $\delta = \frac{1}{9}$. An illustration of the discontinuous function $a$ is provided in Figure \ref{Fig: plota}.

\begin{figure}
	\begin{subfigure}{0.49\textwidth}
		\includegraphics[width=\textwidth]{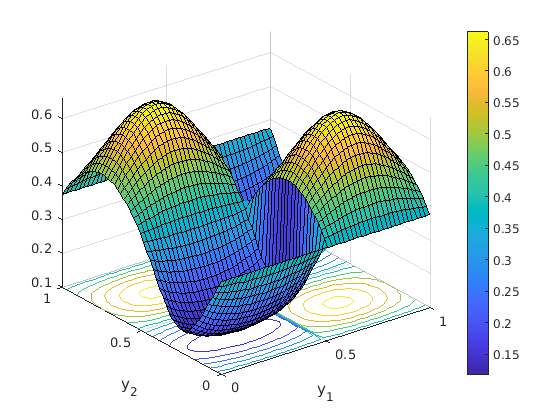}
		\subcaption{$a(y_1,y_2)$ for $(y_1,y_2)\in Y$.}
	\end{subfigure}
	\begin{subfigure}{0.49\textwidth}
		\includegraphics[width=\textwidth]{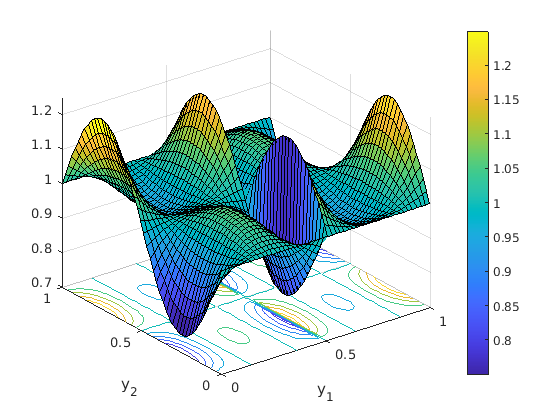}
		\subcaption{$r(y_1,y_2)$ for $(y_1,y_2)\in Y$.}
	\end{subfigure}
	\caption{Plot of the function $a$ defined in \eqref{A for numexp} and plot of the invariant measure $r$ to $A=\mathrm{diag}(1-a,a)$ given by \eqref{r fornumexp}.}
	\label{Fig: plota}
\end{figure}

By Theorem \ref{Thm: C+aM}, noting that $A = C+aM$ with $C:=\mathrm{diag}(1,0)$ and $M:=\mathrm{diag}(-1,1)$, the invariant measure to $A$ is given by 
\begin{align}\label{r fornumexp}
r(y) = 1+M:D^2 [T(A,a)](y) = 1 + \frac{1}{8}(\pi^2 \theta(y_1)-2\omega (y_1))\sin(2\pi y_2)
\end{align}
for $y=(y_1,y_2)\in \R^2$, where we have used that $T(A,a)\in H^2_{\mathrm{per},0}(Y)$ (recall Definition \ref{Def: T(A,a)}) is given by $[T(A,a)](y) = -\frac{1}{32}\theta(y_1)\sin(2\pi y_2)$ for $y=(y_1,y_2)\in \R^2$. The effective coefficient matrix $\bar{A}:=\int_Y rA$ is then given by
\begin{align*}
\bar{A} = C + (r,a)_{L^2(Y)} M = C + \frac{3}{8} M = \frac{1}{8}\mathrm{diag}(5,3).
\end{align*} 
A plot of the discontinuous function $r$ is provided in Figure \ref{Fig: plota}.

In our numerical experiment, we use FreeFem\texttt{++} \cite{Hec12} to approximate $r$ by Theorem \ref{Thm: approx of r} in combination with the method from Lemma \ref{Lmm: Approx rtilde mixed}, where we choose $P_h$ to be the finite-dimensional subspace of $H^1_{\mathrm{per}}(Y;\R^2)$ consisting of vector fields whose components are continuous $Y$-periodic piecewise affine functions with zero mean over $Y$ on a periodic shape-regular triangulation $\mathcal{T}_h$ of $\overline{Y}$ into triangles with vertices $\{(ih,jh)\}_{1\leq i,j\leq N}$ where $N = \frac{1}{h}\in \N$. The integrals in the computation of $c_h$ in \eqref{rh defn} and $\bar{A}_h$ in \eqref{Abar h def} have been obtained using the default quadrature formula in FreeFem\texttt{++} on a fine mesh. The results are presented in Figure \ref{Fig: err curves}. 

By Lemma \ref{Lmm: Approx rtilde mixed} and Theorem \ref{Thm: approx of r} we have that $\|r-r_h\|_{L^2(Y)}\leq C \inf_{w_h\in P_h}\|D(p-w_h)\|_{L^2(Y)}$ for some constant $C>0$, where $p\in H^1_{\mathrm{per},0}(Y;\R^2)$ is the function from Lemma \ref{Lmm: Approx rtilde mixed}(i). For the approximation of $r$, we observe convergence of order $\calO(h^{s})$ in the $L^2(Y)$-norm for some $s\in (0,1)$ in Figure \ref{Fig: err curves}. This indicates that $p\in H^{1+s}(Y)$ in which case we have by standard interpolation inequalities that $\inf_{w_h\in P_h}\|D(p-w_h)\|_{L^2(Y)}\leq Ch^s \|p\|_{H^{1+s}(Y)}$. In Figure \ref{Fig: err curves}, we observe the superconvergence $\|r-r_h\|_{L^2(Y)} = \calO(h)$ for $h = 2^{-i}$, $i\in \N$, i.e., when there is no element of the triangulation whose interior intersects the line $\{y_1 = \frac{1}{2}\}$ along which $r$ has a jump.

For the approximation of $\bar{A}$, we observe that $\lvert \bar{A}-\bar{A}_h\rvert = \calO(h)$, and the superconvergence $\lvert \bar{A}-\bar{A}_h\rvert = \calO(h^2)$ for $h = 2^{-i}$, $i\in \N$.

\begin{figure}
	\begin{subfigure}{0.49\textwidth}
		\includegraphics[width=\textwidth]{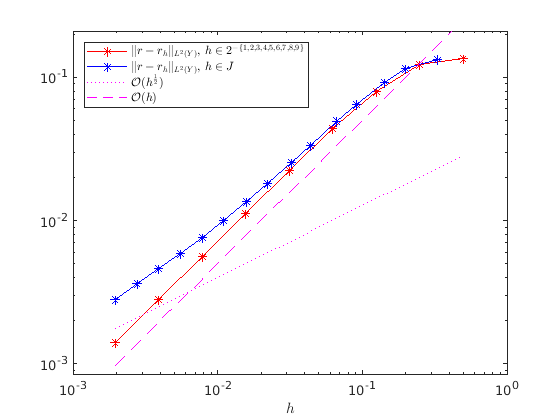}
		\subcaption{$\|r-r_h\|_{L^2(Y)}$}
	\end{subfigure}
	\begin{subfigure}{0.49\textwidth}
		\includegraphics[width=\textwidth]{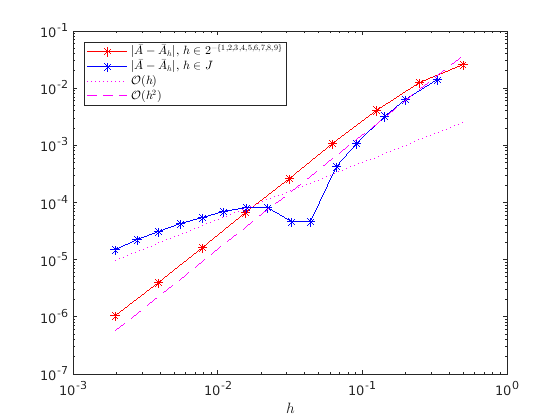}
		\subcaption{$\lvert \bar{A}-\bar{A}_h\rvert$}
	\end{subfigure}
	\caption{Approximation error for the approximation of the invariant measure $r$ and the effective coefficient matrix $\bar{A}$ from Section \ref{Sec: Exp}. We observe two curves, corresponding to whether or not there are elements of the triangulation whose interior intersect the line $\{y_1 = \frac{1}{2}\}$, i.e., the set along which $r$ exhibits a jump.}
	\label{Fig: err curves}
\end{figure}

\section{Collection of the Proofs}\label{Sec: Pfs}

\subsection{Proofs for Section \ref{Sec: Framework}}

\begin{proof}[Proof of Theorem \ref{Thm: Well-posedness of uepsproblem}]

First, note that there exists a constant $\tilde{c}_0 = \tilde{c}_0(\mathrm{diam}(\Omega),n)>0$ such that $\|v\|_{H^2(\Omega)}\leq \tilde{c}_0 \|\Delta v\|_{L^2(\Omega)}$ for any $v\in H^2(\Omega)\cap H^1_0(\Omega)$. Since $\tilde{f}_{\eps}:=f-L_{\eps}g\in L^2(\Omega)$, we know from Theorem 3 in \cite{SS13} that there exists a unique function $\tilde{u}_{\eps}\in H^2(\Omega)\cap H^1_0(\Omega)$ such that $L_{\eps} \tilde{u}_{\eps} = \tilde{f}_{\eps}$ a.e. in $\Omega$, and we have that 
\begin{align*}
\|\tilde{u}_{\eps}\|_{H^2(\Omega)} \leq \tilde{c}_0 C_{\delta}\|\gamma^{\eps}\tilde{f}_{\eps}\|_{L^2(\Omega)},
\end{align*}
where $C_{\delta}:=(1-\sqrt{1-\delta})^{-1}$ and $\gamma^{\eps}:=\gamma(\frac{\cdot}{\eps})$ with $\gamma$ defined in Remark \ref{Rk: Cordes holds}. As $\|\gamma\|_{L^\infty(\R^n)}\leq \frac{\Lambda}{\lambda^2}$ and $\|L_{\eps}g\|_{L^2(\Omega)}\leq c_1 \|D^2 g\|_{L^2(\Omega)}$ with $c_1:= \sqrt{n}\Lambda$, it follows that $\|\tilde{u}_{\eps}\|_{H^2(\Omega)}\leq c_2(\|f\|_{L^2(\Omega)}+\|g\|_{H^2(\Omega)})$ with $c_2:=\tilde{c}_0(1+c_1)C_{\delta}\frac{\Lambda}{\lambda^2}$. We find that $u_{\eps}:=\tilde{u}_{\eps} + g \in H^2(\Omega)$ is the unique solution to \eqref{ueps problem} in $H^2(\Omega)$ and we have the bound $\|u_{\eps}\|_{H^2(\Omega)}\leq \|\tilde{u}_{\eps}\|_{H^2(\Omega)} + \|g\|_{H^2(\Omega)}\leq (c_2+1)(\|f\|_{L^2(\Omega)}+\|g\|_{H^2(\Omega)})$. 
\end{proof}

\subsection{Proofs for Section \ref{Sec: Prelim}}

\begin{proof}[Proof of Lemma \ref{Lmm: b_mu}]
In view of Remark \ref{Rk: Cordes holds} and Remark \ref{Rk: norm on H^2_per,0} we have for any $\fhi\in H^2_{\mathrm{per}}(Y)$ that
\begin{align*}
b_{\mu}(\fhi,\fhi) &= \|\mu\fhi-\Delta \fhi\|_{L^2(Y)}^2 + \left((I_n-\gamma A):D^2 \fhi, \mu\fhi-\Delta \fhi\right)_{L^2(Y)} \\&\geq \|\mu\fhi-\Delta \fhi\|_{L^2(Y)}^2-\sqrt{1-\delta}\|D^2 \fhi\|_{L^2(Y)} \|\mu\fhi-\Delta \fhi\|_{L^2(Y)}.  
\end{align*} 
Finally, using that for any $\fhi\in H^2_{\mathrm{per}}(Y)$ we have that
\begin{align*}
\|D^2\fhi\|_{L^2(Y)}^2 = \|\Delta \fhi\|_{L^2(Y)}^2 \leq \mu^2 \|\fhi\|_{L^2(Y)}^2 +2 \mu \|\nabla \mu\|_{L^2(Y)}^2 + \|\Delta \fhi\|_{L^2(Y)}^2 =  \|\mu\fhi-\Delta \fhi\|_{L^2(Y)}^2,
\end{align*}
the claimed result follows.
\end{proof}

\subsection{Proofs for Section \ref{Sec: Per problems}}

\begin{proof}[Proof of Theorem \ref{Thm: invmeas}]
(i) Note that $b_0$ is coercive on $H^2_{\mathrm{per},0}(Y)$ by Lemma \ref{Lmm: b_mu} and that 
\begin{align}\label{boundedness b_0}
\lvert b_0(\fhi_1,\fhi_2)\rvert = \lvert (\gamma A:D^2 \fhi_1,\Delta \fhi_2)_{L^2(Y)}\rvert \leq \sqrt{n}\frac{\Lambda}{\lambda}\|\Delta \fhi_1\|_{L^2(Y)}\|\Delta\fhi_2\|_{L^2(Y)}\quad\forall \fhi_1,\fhi_2\in H^2_{\mathrm{per},0}(Y), 
\end{align}
where we have used Remark \ref{Rk: norm on H^2_per,0} and that $\gamma \lvert A\rvert = \frac{\mathrm{tr}(A)}{\lvert A\rvert} \leq  \frac{n\Lambda}{\sqrt{n}\lambda} = \sqrt{n}\frac{\Lambda}{\lambda}$ almost everywhere. Since $l:H^2_{\mathrm{per},0}(Y)\rightarrow \R$, $l(\fhi):= \int_Y \tilde{A}:D^2 \fhi$ is a bounded linear map with $\lvert l(\fhi)\rvert \leq \sqrt{n}\frac{\Lambda}{\lambda}\|\Delta\fhi\|_{L^2(Y)}$ for any $\fhi\in H^2_{\mathrm{per},0}(Y)$, we deduce from the Lax--Milgram theorem that there exists a unique $\psi\in H^2_{\mathrm{per},0}(Y)$ such that $b_0(\fhi,\psi) = l(\fhi)$ for any $\fhi\in H^2_{\mathrm{per},0}(Y)$ and we have the bound 
\begin{align*}
\|\Delta \psi\|_{L^2(Y)}^2 \leq C_{\delta}\, b_0(\psi,\psi) = C_{\delta}\, l(\psi)\leq \sqrt{n}\frac{\Lambda}{\lambda}C_{\delta}\|\Delta \psi\|_{L^2(Y)}.
\end{align*}

(ii) Since $\tilde{r}:=1-\Delta \psi \in L^2_{\mathrm{per}}(Y)$, $\int_Y \tilde{r} = 1$, and $(\tilde{r},\tilde{A}:D^2 \fhi)_{L^2(Y)} = l(\fhi)-b_0(\fhi,\psi) = 0$ for any $\fhi\in H^2_{\mathrm{per}}(Y)$, we have that $\tilde{r}$ is an invariant measure to $\tilde{A}$. Suppose $\hat{r}\in L^2_{\mathrm{per}}(Y)$ is another invariant measure to $\tilde{A}$. Then, as $\int_Y \hat{r} = 1$, there exists a unique $\xi\in H^2_{\mathrm{per},0}(Y)$ such that $\Delta \xi = 1-\hat{r}$. Since $b_0(\fhi,\xi) = l(\fhi)-(\hat{r},\tilde{A}:D^2 \fhi)_{L^2(Y)} = l(\fhi)$ for any $\fhi\in H^2_{\mathrm{per}}(Y)$, we have that $\xi = \psi$ and thus, $\hat{r} =  1-\Delta \psi = \tilde{r}$. Therefore, $\tilde{r}$ is the unique invariant measure to $\tilde{A}$.

It remains to show that $\tilde{r}\geq 0$ almost everywhere. For $k\in \N$ we set $a_{k,ij}:= a_{ij}\ast w_k$ for $1\leq i,j\leq n$, where $w_k:=k^n w(k\,\cdot)$ for some $w\in C_c^{\infty}(\R^n)$ with $w\geq 0$ in $\R^n$ and $\int_{\R^n} w = 1$. Then, $A_k:=(a_{k,ij})_{1\leq i,j\leq n}\in C^{\infty}_{\mathrm{per}}(Y;\R^{n\times n}_{\mathrm{sym}})\cap \mathcal{M}(\lambda,\Lambda)$ for all $k\in \N$ and we have that $\lim_{k\rightarrow\infty}\|A_k-A\|_{L^2(Y)}= 0$. We set $\gamma_k := \frac{\mathrm{tr}(A_k)}{\lvert A_k\rvert^2}\in C^{\infty}_{\mathrm{per}}(Y)$ and note that $\frac{\lambda}{\Lambda^2}\leq \gamma_k\leq \frac{\Lambda}{\lambda^2}$ for any $k\in \N$, as observed in Remark \ref{Rk: Cordes holds}. Let $\tilde{r}_k\in L^2_{\mathrm{per}}(Y)$ be the unique invariant measure to $\tilde{A}_k:=\gamma_k A_k$. Since $\tilde{A}_k\in C^{\infty}_{\mathrm{per}}(Y)$, we have that $\tilde{r}_k\in C^{\infty}_{\mathrm{per}}(Y)$ and it is known (see e.g., \cite{BLP11}) that $\tilde{r}_k > 0$ in $\R^n$. Setting $q_k := \tilde{r}_k\gamma_k\in C^{\infty}_{\mathrm{per}}(Y)$, we have that $q_k\geq \frac{\lambda}{\Lambda^2}\tilde{r}_k > 0$ in $\R^n$ and, by (i) and the first part of (ii), we find $\|q_k\|_{L^2(Y)}\leq \frac{\Lambda}{\lambda^2}\|\tilde{r}_k\|_{L^2(Y)}\leq C$ for all $k\in \N$ for some constant $C=C(\lambda,\Lambda,n,\delta)>0$. Therefore, there exists $q\in L^2_{\mathrm{per}}(Y)$ such that, upon passing to a subsequence, $q_k\weak q$ weakly in $L^2(Y)$. Thus, for any $\fhi\in C^{\infty}_{\mathrm{per}}(Y)$ we have
\begin{align*}
0 = (q_k, A_k:D^2 \fhi)_{L^2(Y)} = (q_k, A:D^2 \fhi)_{L^2(Y)} + (q_k, (A_k-A):D^2 \fhi)_{L^2(Y)}  \underset{k\rightarrow \infty}{\longrightarrow} (q, A:D^2 \fhi)_{L^2(Y)}.
\end{align*}
It follows that $q\in L^2_{\mathrm{per}}(Y)$ is a solution to $-D^2:(q A) = 0$ in $Y$. Note that $q\geq 0$ a.e. in $\R^n$ and $\int_Y q \geq \frac{\lambda}{\Lambda^2}>0$ as $q_k > 0$ in $\R^n$ and $\int_Y q_k \geq \frac{\lambda}{\Lambda^2}\int_Y \tilde{r}_k =  \frac{\lambda}{\Lambda^2}$ for any $k\in\N$. Recalling that $\gamma > 0$ a.e. in $\R^n$, we see that $\int_Y \frac{q}{\gamma} = \|\frac{q}{\gamma}\|_{L^1(Y)} > 0$. Setting $\tilde{q} := (\int_Y \frac{q}{\gamma})^{-1} \frac{q}{\gamma}\in L^2_{\mathrm{per}}(Y)$ and using uniqueness of the invariant measure to $\tilde{A}$, we see that $\tilde{r}=\tilde{q}\geq 0$ a.e. in $\R^n$.

(iii) This follows immediately from (ii) and the fact that $\frac{\lambda}{\Lambda^2}\leq \gamma\leq \frac{\Lambda}{\lambda^2}$ a.e. in $\R^n$. Note that the bound \eqref{r bound} follows from $\|r\|_{L^2(Y)} \leq \frac{\Lambda}{\lambda^2}  ( \gamma,\tilde{r} )_{L^2(Y)}^{-1}\|\tilde{r}\|_{L^2(Y)}$ and $( \gamma,\tilde{r} )_{L^2(Y)}\geq \frac{\lambda}{\Lambda^2}\int_Y \tilde{r} = \frac{\lambda}{\Lambda^2}$, where we have used that $\tilde{r}\geq 0$ almost everywhere and $\int_Y \tilde{r} = 1$.
\end{proof}

\begin{proof}[Proof of Theorem \ref{Thm: Main thm double div}]
(i) If $q\in L^2_{\mathrm{per}}(Y)$ is a solution to \eqref{q equation}, then $\int_Y f = (f,1)_{L^2(Y)} = (q,-A:D^2 1)_{L^2(Y)}= 0$, where $1$ denotes the constant function with value one. Now suppose that $\int_Y f = 0$. Then, by Lemma \eqref{Lmm: b_mu} and the Lax--Milgram theorem, there exists a unique $\eta\in H^2_{\mathrm{per},0}(Y)$ such that $b_0(\fhi,\eta)=(f,\fhi)_{L^2(Y)}$ for any $\fhi\in H^2_{\mathrm{per},0}(Y)$. Equivalently, since $\int_Y f = 0$, we have $b_0(\fhi,\eta)=(f,\fhi)_{L^2(Y)}$ for any $\fhi\in H^2_{\mathrm{per}}(Y)$. Setting $q := -\gamma\Delta\eta\in L^2_{\mathrm{per}}(Y)$, we have that $(q,-A:D^2 \fhi)_{L^2(Y)} = b_0(\fhi,\eta)=(f,\fhi)_{L^2(Y)}$ for any $\fhi\in H^2_{\mathrm{per}}(Y)$, i.e., $q$ is a solution to \eqref{q equation}.

(ii) Suppose that $q_1,q_2\in L^2_{\mathrm{per}}(Y)$ are solutions to \eqref{q equation}. Then, $w:=q_1-q_2\in L^2_{\mathrm{per}}(Y)$ satisfies $(w,-A:D^2 \fhi)_{L^2(Y)} = 0$ for all $\fhi\in H^2_{\mathrm{per}}(Y)$ and $\int_Y w = c$. This gives $w = cr$. Indeed, by uniqueness of $r$, we see that if $c\neq 0$ then $r = c^{-1} w$, and if $c = 0$ then $r = w+r$, i.e., $w = 0 = cr$.

(iii) Note that $q_0 = -\gamma\Delta \eta + (\gamma, \Delta \eta)_{L^2(Y)}r$, where $\eta\in H^2_{\mathrm{per},0}(Y)$ is as in the proof of (i). By Lemma \ref{Lmm: b_mu} and Remark \ref{Rk: norm on H^2_per,0} we have $\|\Delta \eta\|_{L^2(Y)}^2 \leq C_{\delta}b_0(\eta,\eta) = C_{\delta}(f,\eta)_{L^2(Y)}\leq \frac{n}{\pi^2}C_{\delta}\|f\|_{L^2(Y)}\|\Delta \eta\|_{L^2(Y)}$. Using that $\frac{\lambda}{\Lambda^2}\leq \gamma\leq \frac{\Lambda}{\lambda^2}$ a.e. in $\R^n$, we see that $\|q_0\|_{L^2(Y)}\leq \frac{\Lambda}{\lambda^2}\left(1+\|r\|_{L^2(Y)}\right) \|\Delta \eta\|_{L^2(Y)}$ which yields the claimed result in combination with $\|\Delta \eta\|_{L^2(Y)}\leq \frac{n}{\pi^2}C_{\delta}\|f\|_{L^2(Y)}$.
\end{proof}

\begin{proof}[Proof of Theorem \ref{Thm: per nondiv}]
(i,ii) If $v\in H^2_{\mathrm{per}}(Y)$ solves \eqref{v equation}, then $(f,r)_{L^2(Y)} = (-A:D^2 v,r)_{L^2(Y)}= 0$. Now suppose $(f,r)_{L^2(Y)} = 0$. By the Lax--Milgram theorem, in view of Lemma \ref{Lmm: b_mu}, there exists a unique $v\in H^2_{\mathrm{per},0}(Y)$ such that  $b_0(v,\fhi)=(\gamma f,-\Delta \fhi)_{L^2(Y)}$ for any $\fhi \in H^2_{\mathrm{per},0}(Y)$. Equivalently, since $-\Delta:H^2_{\mathrm{per},0}(Y)\rightarrow L^2_{\mathrm{per},0}(Y)$ is a bijection, we have that
\begin{align}\label{pf of per A:D^2 1}
(-\gamma A:D^2 v,\phi)_{L^2(Y)} = (\gamma f,\phi)_{L^2(Y)}\qquad \forall \phi\in L^2_{\mathrm{per},0}(Y).
\end{align}
We claim that this is equivalent to
\begin{align}\label{pf of per A:D^2 2}
(-\gamma A:D^2 v,\phi)_{L^2(Y)} = (\gamma f,\phi)_{L^2(Y)}\qquad \forall \phi\in L^2_{\mathrm{per}}(Y).
\end{align}
Indeed, to see that \eqref{pf of per A:D^2 1} implies \eqref{pf of per A:D^2 2}, we write $\phi\in  L^2_{\mathrm{per}}(Y)$ as $\phi = \phi_1 + \phi_2$ with $\phi_1:= \phi - c\tilde{r}\in L^2_{\mathrm{per},0}(Y)$ and $\phi_2:=c\tilde{r}$, where $c:=\int_Y \phi$ and $\tilde{r}:=(\frac{1}{\gamma},r)_{L^2(Y)}^{-1} \frac{r}{\gamma}\in L^2_{\mathrm{per}}(Y)$ is the invariant measure to $\tilde{A}:=\gamma A$, and use that $(-\tilde{A}:D^2 v,\phi_2)_{L^2(Y)} = 0 = c(\frac{1}{\gamma},r)_{L^2(Y)}^{-1} (f,r)_{L^2(Y)}  = (\gamma f,\phi_2)_{L^2(Y)}$. Noting that \eqref{pf of per A:D^2 2} is equivalent to $-\gamma A:D^2 v = \gamma f$ a.e. in $\R^n$, which in turn is equivalent to $-A:D^2 v = f$ a.e. in $\R^n$ since $\gamma\geq \frac{\lambda}{\Lambda^2}>0$ a.e. in $\R^n$, we immediately obtain (i)--(ii).  

(iii) Note that $v_0$ is the unique element in $H^2_{\mathrm{per},0}(Y)$ such that $b_0(v_0,\fhi)=-(\gamma f,\Delta \fhi)_{L^2(Y)}$ for any $\fhi\in H^2_{\mathrm{per}}(Y)$. Hence, $\|\Delta v_0\|_{L^2(Y)}^2 \leq C_{\delta}b_0(v_0,v_0) = -C_{\delta}(\gamma f,\Delta v_0)_{L^2(Y)}\leq \frac{\Lambda}{\lambda^2}C_{\delta}\|f\|_{L^2(Y)}\|\Delta v_0\|_{L^2(Y)}$ by Lemma \ref{Lmm: b_mu} and Remark \ref{Rk: Cordes holds}.
\end{proof}

\subsection{Proofs for Section \ref{Sec: conv result}}

\begin{proof}[Proof of Theorem \ref{Thm: Conv}]
First, recall that by the uniform $H^2$-bound \eqref{uniform H2 bd} there exists a function $u\in H^2(\Omega)$ with $u-g\in H^1_0(\Omega)$ such that, upon passing to a subsequence, \eqref{limit for subs} holds. Passing to the limit in \eqref{divform eqn} and using \eqref{lim of M} we find that $u$ is a solution to \eqref{u problem}. Noting that 
\begin{align*}
\lambda\lvert \xi\rvert^2 = \lambda\lvert \xi\rvert^2 (r,1)_{L^2(Y)}  \leq (r,A\xi\cdot \xi)_{L^2(Y)} =  \bar{A}\xi\cdot \xi  \leq \Lambda\lvert \xi\rvert^2 (r,1)_{L^2(Y)} = \Lambda\lvert \xi\rvert^2\quad\forall \xi\in \R^n,
\end{align*}
where we have used that $r\geq 0$ almost everywhere and $A\in \mathcal{M}(\lambda,\Lambda)$, we see that $u$ is the unique solution in $H^2(\Omega)$ to \eqref{u problem}. We conclude that the whole sequence $(u_{\eps})_{\eps > 0}$ convergences weakly in $H^2(\Omega)$ to $u$.
\end{proof}

\subsection{Proofs for Section \ref{Sec: corrector est}}

\begin{proof}[Proof of Theorem \ref{Thm: corrI}]
We have $r_{\eps}:=u_{\eps} - u - \eps^2 ( \eta_{\eps} + \theta_{\eps})\in H^2(\Omega)\cap H^1_0(\Omega)$ and $L_{\eps}r_{\eps} = \eps A^{\eps}:P_{\eps}$. By the proof of Theorem \ref{Thm: Well-posedness of uepsproblem} and using that $\gamma \lvert A\rvert = \frac{\mathrm{tr}(A)}{\lvert A\rvert} \leq  \frac{n\Lambda}{\sqrt{n}\lambda} = \sqrt{n}\frac{\Lambda}{\lambda}$ almost everywhere, we then have the bound
\begin{align*}
\|r_{\eps}\|_{H^2(\Omega)}\leq C_{\delta}\tilde{c}_0 \|\gamma^{\eps}L_{\eps}r_{\eps}\|_{L^2(\Omega)}\leq \eps C_{\delta}\tilde{c}_0\|\gamma\lvert A\rvert \|_{L^{\infty}(\R^n)} \|P_{\eps}\|_{L^2(\Omega)}\leq \eps  C_{\delta}\tilde{c}_0\sqrt{n}\frac{\Lambda}{\lambda} \|P_{\eps}\|_{L^2(\Omega)},
\end{align*}
where $C_{\delta}:=(1-\sqrt{1-\delta})^{-1}$ and $\tilde{c}_0 = \tilde{c}_0(\mathrm{diam}(\Omega),n)>0$ is as in the proof of Theorem \ref{Thm: Well-posedness of uepsproblem}. The claim follows with $c_0:= \tilde{c}_0\sqrt{n}$.
\end{proof}

\begin{proof}[Proof of Theorem \ref{Thm: CorrII}]
We set $d_{\eps}:=r_{\eps}+2\eps( z_{\eps}-\eps^2(\tilde{\eta}_{\eps}+\tilde{\theta}_{\eps}))$ where $r_{\eps}:=u_{\eps} - u - \eps^2 ( \eta_{\eps} + \theta_{\eps})$. We have that $d_{\eps}\in H^2(\Omega)\cap H^1_0(\Omega)$ and, writing $c_j^{kl}:=(Ae_j\cdot \nabla v_{kl},r)_{L^2(Y)}$ for $1\leq j,k,l\leq n$, there holds
\begin{align*}
L_{\eps} d_{\eps} &= \eps A^{\eps}:P_{\eps}+2\eps\left(-\sum_{j,k,l=1}^n c_j^{kl}\partial^3_{jkl} u + \eps^2 A^{\eps}:D^2\tilde{\eta}_{\eps}   \right)\\
&= \eps A^{\eps}:P_{\eps}-2\eps\sum_{j,k,l=1}^n [c_j^{kl}-A:D^2 \chi_{jkl}]^{\eps}\partial^3_{jkl} u + 2\eps^2 \sum_{i,j=1}^n a_{ij}^{\eps}(2[\partial_i X]^{\eps}:D^3(\partial_j u) + \eps X^{\eps}:D^3(\partial^2_{ij}u)  )    \\
&= \eps^2 A^{\eps}:Q_{\eps},
\end{align*}
where we have used that $\chi_{jkl}=T(A,Ae_j\cdot \nabla v_{kl})$ and $L_{\eps}r_{\eps} = \eps A^{\eps}:P_{\eps}$ with $P_{\eps}$ defined as in Theorem \ref{Thm: corrI}. By the proof of Theorem \ref{Thm: Well-posedness of uepsproblem} and using that $\gamma \lvert A\rvert = \frac{\mathrm{tr}(A)}{\lvert A\rvert} \leq  \frac{n\Lambda}{\sqrt{n}\lambda} = \sqrt{n}\frac{\Lambda}{\lambda}$ almost everywhere, we then have the bound
\begin{align*}
\|d_{\eps}\|_{H^2(\Omega)}\leq C_{\delta}\tilde{c}_0 \|\gamma^{\eps} L_{\eps}d_{\eps}\|_{L^2(\Omega)}\leq \eps^2 C_{\delta}\tilde{c}_0\|\gamma\lvert A\rvert \|_{L^{\infty}(\R^n)} \|Q_{\eps}\|_{L^2(\Omega)}\leq \eps^2  C_{\delta}\tilde{c}_0\sqrt{n}\frac{\Lambda}{\lambda} \|Q_{\eps}\|_{L^2(\Omega)},
\end{align*}
where $C_{\delta}:=(1-\sqrt{1-\delta})^{-1}$ and $\tilde{c}_0 = \tilde{c}_0(\mathrm{diam}(\Omega),n)>0$ is as in the proof of Theorem \ref{Thm: Well-posedness of uepsproblem}. The claim follows with $c_0:= \tilde{c}_0\sqrt{n}$.
\end{proof}

\subsection{Proofs for Section \ref{Sec: Type}}

The proofs of Theorem \ref{Thm: C+aM} and Theorem \ref{Thm: type-eps^2 diag} are almost identical to the proofs in \cite{GST22} and hence omitted.

\subsection{Proofs for Section \ref{Sec: Num methods}}

\begin{proof}[Proof of Lemma \ref{Lmm: H2 approx rtilde}]
In view of Lemma \ref{Lmm: b_mu}, we note that existence and uniqueness of $\psi_h\in \Psi_h$ satisfying \eqref{psi h prob} follows from the Lax--Milgram theorem. Noting that $\psi-\psi_h\in H^2_{\mathrm{per},0}(Y)$ and using Lemma \ref{Lmm: b_mu}, Galerkin orthogonality, and \eqref{boundedness b_0}, we have for any $\fhi_h\in \Psi_h$ that
\begin{align*}
C_{\delta}^{-1}\|\Delta(\psi-\psi_h)\|_{L^2(Y)}^2 &\leq b_0(\psi - \psi_h,\psi-\psi_h) \\ &= b_0(\psi - \fhi_h,\psi-\psi_h)\leq \sqrt{n}\frac{\Lambda}{\lambda}\|\Delta(\psi - \fhi_h)\|_{L^2(Y)}\|\Delta(\psi-\psi_h)\|_{L^2(Y)},
\end{align*}
which yields the claimed result. Note that $\tilde{r}-\tilde{r}_h = -\Delta(\psi-\psi_h)$ since $\tilde{r}_h = 1-\Delta \psi_h$ and $\tilde{r} = 1-\Delta \psi$ by Theorem \ref{Thm: invmeas}.
\end{proof}

\begin{proof}[Proof of Lemma \ref{Lmm: Approx rtilde mixed}]
First, we note that for any $w\in H^1_{\mathrm{per}}(Y;\R^n)$ there holds
\begin{align*}
\|\nabla\cdot w\|_{L^2(Y)}^2 + \frac{1}{2}\|Dw-Dw^{\mathrm{T}}\|_{L^2(Y)}^2 = \|Dw\|_{L^2(Y)}^2 + \|\nabla\cdot w\|_{L^2(Y)}^2 - (Dw,Dw^{\mathrm{T}})_{L^2(Y)} = \|Dw\|_{L^2(Y)}^2,
\end{align*}
where the second equality follows from integration by parts and a density argument. We find that
\begin{align}\label{coer a}
\begin{split}
b(w,w) &= \|D w\|_{L^2(Y)}^2 + ((\tilde{A}-I_n):Dw,\nabla\cdot w)_{L^2(Y)} \\
&\geq \|D w\|_{L^2(Y)}^2 - \sqrt{1-\delta}\|\nabla\cdot w\|_{L^2(Y)}\|Dw\|_{L^2(Y)} \geq C_{\delta}^{-1}\|Dw\|_{L^2(Y)}^2
\end{split}
\end{align}
for any $w\in H^1_{\mathrm{per}}(Y;\R^n)$, where $C_{\delta} = (1-\sqrt{1-\delta})^{-1}$. Further, using that $\|\gamma\lvert A\rvert\|_{L^{\infty}(\R^n)}\leq \sqrt{n}\frac{\Lambda}{\lambda}$ and setting $C_{n,\lambda,\Lambda}:=1+\sqrt{n}\frac{\Lambda}{\lambda}$, we have for any $w\in H^1_{\mathrm{per}}(Y;\R^n)$ that
\begin{align}\label{boun a}
\begin{split}
\lvert b(w,\tilde{w})\rvert &\leq \sqrt{n}\frac{\Lambda}{\lambda} \|Dw\|_{L^2(Y)}\|\nabla\cdot \tilde{w}\|_{L^2(Y)} + \frac{1}{2}\|Dw-Dw^{\mathrm{T}}\|_{L^2(Y)}\|D\tilde{w}-D\tilde{w}^{\mathrm{T}}\|_{L^2(Y)}\\ &\leq C_{n,\lambda,\Lambda}\|Dw\|_{L^2(Y)}\|D\tilde{w}\|_{L^2(Y)},
\end{split}
\end{align}
where we have used in the second inequality that $\frac{1}{2}\|Dw-Dw^{\mathrm{T}}\|_{L^2(Y)}^2\leq \|Dw\|_{L^2(Y)}^2$ for any $w\in H^1_{\mathrm{per}}(Y;\R^n)$. Since $w\mapsto \|Dw\|_{L^2(Y)}$ defines a norm on $H^1_{\mathrm{per},0}(Y;\R^n)$, it follows that $b$ defines a bounded coercive bilinear form on $H^1_{\mathrm{per},0}(Y;\R^n)$. By the Lax--Milgram theorem, using that $w\mapsto (\gamma,A:Dw)_{L^2(Y)}$ defines a bounded linear functional on $H^1_{\mathrm{per},0}(Y;\R^n)$, we have that there exist unique $p\in H^1_{\mathrm{per},0}(Y;\R^n)$ and $p_h\in P_h$ such that
\begin{align*}
b(w,p) = (\gamma,A:Dw)_{L^2(Y)}\quad\forall w\in H^1_{\mathrm{per},0}(Y;\R^n),\qquad b(w_h,p_h) = (\gamma,A:Dw_h)_{L^2(Y)}\quad\forall w_h\in P_h.
\end{align*}
Noting that for any $\fhi\in H^2_{\mathrm{per}}(Y)$ we have that $\nabla \fhi\in H^1_{\mathrm{per},0}(Y;\R^n)$, we see that
\begin{align*}
(1-\nabla\cdot p,-\tilde{A}:D^2\fhi)_{L^2(Y)} = b(\nabla \fhi,p) - (\gamma,A:D^2 \fhi)_{L^2(Y)}  = 0\qquad\forall\fhi\in H^2_{\mathrm{per}}(Y).
\end{align*}
Therefore, using that $\nabla\cdot p\in L^2_{\mathrm{per},0}(Y)$, we have that $\tilde{r} = 1-\nabla\cdot p$. It only remains to prove (iii). Noting that $p-p_h\in H^1_{\mathrm{per},0}(Y;\R^n)$ and using \eqref{coer a}, Galerkin orthogonality, and \eqref{boun a}, we have that
\begin{align*}
C_{\delta}^{-1}\|D(p-p_h)\|_{L^2(Y)}^2 &\leq b(p-p_h,p-p_h) \\ &= b(p - w_h,p-p_h)\leq C_{n,\lambda,\Lambda}\|D(p - w_h)\|_{L^2(Y)}\|D(p-p_h)\|_{L^2(Y)}
\end{align*}
for any $w_h\in P_h$. Since $\|\nabla\cdot (p-p_h)\|_{L^2(Y)}\leq \|D(p-p_h)\|_{L^2(Y)}$, this yields the claimed result. Note that $\tilde{r}-\tilde{r}_h = -\nabla\cdot(p-p_h)$ since $\tilde{r}_h = 1-\nabla\cdot p_h$ and $\tilde{r} = 1-\nabla\cdot p$. 
\end{proof}

\begin{proof}[Proof of Theorem \ref{Thm: approx of r}]
First, recall that $r = c^{-1}\gamma \tilde{r}$ with $c:=(\gamma,\tilde{r})_{L^2(Y)}\in [ \frac{\lambda}{\Lambda^2},\frac{\Lambda}{\lambda^2}]$, where we have used the bounds on $\gamma$ from Remark \ref{Rk: Cordes holds}, $\int_Y \tilde{r} = 1$, and $\tilde{r}\geq 0$ almost everywhere. Noting that $\lvert c - c_h\rvert \leq \frac{\Lambda}{\lambda^2} \|\tilde{r}-\tilde{r}_h\|_{L^2(Y)}$, we see that $c_h\rightarrow c$ and hence, $r_h$ is well-defined for $h>0$ sufficiently small. Further, for $h>0$ sufficiently small, there holds
\begin{align*}
\|r-r_h\|_{L^2(Y)} = \| c^{-1}\gamma \tilde{r}-c_h^{-1}\gamma \tilde{r}_h\|_{L^2(Y)}\leq c_h^{-1}\|\gamma(\tilde{r}-\tilde{r}_h)\|_{L^2(Y)} + \lvert c^{-1}-c_h^{-1}\rvert\, \|\gamma \tilde{r}\|_{L^2(Y)},
\end{align*}
and using that $\gamma \tilde{r} = cr$, we find
\begin{align*}
\|r-r_h\|_{L^2(Y)} &\leq c_h^{-1}\left(\|\gamma(\tilde{r}-\tilde{r}_h)\|_{L^2(Y)} + \lvert c-c_h\rvert\, \|r\|_{L^2(Y)}\right)\\ &\leq \frac{\Lambda}{\lambda^2}c_h^{-1}\left(\|\tilde{r}-\tilde{r}_h\|_{L^2(Y)} + \frac{\Lambda}{\lambda^2}\|\tilde{r}-\tilde{r}_h\|_{L^2(Y)} \|r\|_{L^2(Y)}\right) 
\end{align*} 
Noting that $c_h^{-1}\leq \frac{\Lambda^2}{\lambda}+1$ for $h> 0$ sufficiently small completes the proof.
\end{proof}

\begin{proof}[Proof of Lemma \ref{Lmm: Approx Abar}]
We have $\lvert \bar{A}-\bar{A}_h\rvert = \left\lvert \int_Y (r-r_h)A\right\rvert \leq \int_Y \lvert r-r_h\rvert \lvert A\rvert \leq \sqrt{n}\Lambda \|r-r_h\|_{L^1(Y)}$, and the claim follows.
\end{proof}

\begin{proof}[Proof of Lemma \ref{Lmm: Approx u}]
The proof of this result is analogous to the proof of Lemma 3.4 in \cite{CSS20} and hence omitted.
\end{proof}

\section*{Acknowledgments}

The author gratefully acknowledges helpful conversations with Professor Yves Capdeboscq (Universit\'{e} de Paris) and Professor Hung V. Tran (University of Wisconsin Madison) during the preparation of this work.

\bibliographystyle{plain}
\bibliography{ref_LinftyA}

\end{document}